\documentclass[12pt, reqno]{amsart}
\usepackage[table]{xcolor}
\usepackage{amsmath,amssymb, amsfonts,amscd,psfrag,graphicx,stmaryrd}
\usepackage{epsfig}
\usepackage{amsthm}
\usepackage{fullpage}
\usepackage{verbatim}
\usepackage{listings}
\usepackage[table]{xcolor} 
\usepackage{mathtools}

\newcommand{\lebn}

\usepackage[misc]{ifsym}
\usepackage{caption}
\usepackage{float}
\usepackage{amsmath}
\usepackage{subfigure}

\theoremstyle{plain}

\newtheorem{theorem}[equation]{Theorem}
\newtheorem{conjecture}[equation]{Conjecture}

\newtheorem{lemma}[equation]{Lemma}

\theoremstyle{definition}

\newtheorem{remark}[equation]{Remark}

\numberwithin{equation}{section}

\usepackage{chngcntr}

\counterwithin*{nthm}{subsection}

\newcommand{\D}{\Delta}

\usepackage[a4paper,left=1.5cm,right=1.5cm,top=1.5cm,bottom=2cm]{geometry}

\allowdisplaybreaks

\usepackage{tikz}
\usepackage{verbatim}
\usetikzlibrary{shapes,shadows,calc}
\usepgflibrary{arrows}
\usetikzlibrary{arrows, decorations.markings, calc, fadings, decorations.pathreplacing, patterns, decorations.pathmorphing, positioning}
\tikzset{nodc/.style={circle,draw=blue!50,fill=pink!80,inner sep=1.6pt}}
\tikzset{nodr/.style={circle,draw=black,fill=green!50!black,inner sep=1.6pt}}
\tikzset{nodel/.style={circle,draw=black,inner sep=2.2pt}}
\tikzset{nodinvisible/.style={circle,draw=white,inner sep=2pt}}
\tikzset{nodpale/.style={circle,draw=gray,fill=gray,inner sep=1.6pt}}
\tikzset{nod1/.style={circle,draw=black,fill=black,inner sep=1pt}}
\tikzset{nod2/.style={circle,draw=black,fill=blue!75!black,inner sep=1.6pt}}
\tikzset{nod2w/.style={circle,draw=black,fill=white,inner sep=1.6pt}}
\tikzset{nodgs/.style={circle,draw=black,dotted,fill=gray,inner sep=1.6pt}}
\tikzset{nod3/.style={circle,draw=black,fill=black,inner sep=1.8pt}}
\tikzset{noddiam/.style={diamond,draw=black,inner sep=2pt}}
\tikzset{nodw/.style={circle,draw=black,inner sep=1.8pt}}
\usepackage[colorlinks, urlcolor=blue, linkcolor=blue, citecolor=blue]{hyperref}

\makeatletter
 \def\@textbottom{\vskip \z@ \@plus 10pt}
 \let\@texttop\relax
\makeatother

\begin{document}

\bibliographystyle{plain}

\title[Total coloring of planar graphs with maximum degree eight]{
A sufficient condition for  planar graphs with maximum degree eight to be totally $9$-colorable}

\author{Zakir Deniz and Hakan Guler}

\address{Department of Mathematics, D\"uzce University, D\"uzce, 81620, T{\"{u}}rk{\.{\.i}}ye.}
\email{zakirdeniz@duzce.edu.tr}

\address{Department of Mathematics, Kastamonu University, Kastamonu, 37150, T{\"{u}}rk{\.{\.i}}ye.}
\email{hakanguler19@gmail.com}

\thanks{The authors are supported by T\" UB\. ITAK, grant no:124F450}
\date{\today}

\begin{abstract}
A total coloring of a graph $G$ is a coloring of the vertices and edges such that  two adjacent or incident elements receive different colors. The minimum number of colors required for a total coloring of a graph $G$ is called the total chromatic number, denoted by $\chi''(G)$. Let  $G$ be a planar graph of maximum degree eight. It is known that  $9\leq \chi''(G) \leq 10$. We here prove that $\chi''(G)=9$ when the graph does not contain any subgraph isomorphic to a $4$-fan.
\end{abstract}
\maketitle

\section{Introduction}

All graphs in this paper are assumed to be simple. For terminology and notation not defined here, we refer the reader to \cite{west}. Let $G$ be a graph. We use $V(G)$, $E(G)$, $F(G)$, and $\Delta(G)$ to denote the vertex set, edge set, face set, and maximum degree of $G$, respectively. When the context is clear, we abbreviate $\Delta(G)$ to $\Delta$.
A total coloring of a planar graph $G$ is a coloring of $V(G) \cup E(G)$ such that any two adjacent or incident elements receive different colors. The minimum number of colors required for a total coloring of $G$ is called the total chromatic number, denoted by $\chi''(G)$. We refer the reader to the comprehensive survey by Geetha et al.~\cite{geetha} for progress on the total chromatic number of graphs.
Behzad \cite{behzad} and Vizing \cite{vizing} independently posed the following conjecture, known as the Total Coloring Conjecture.

\begin{conjecture}\label{conj:main}
For any graph $G$, $\chi''(G) \leq \Delta(G) + 2$.
\end{conjecture}

Conjecture \ref{conj:main} has been proved by Rosenfeld \cite{rosenfeld} for $\Delta = 3$, and by Kostochka \cite{kostachka} for $\Delta \leq 5$. For planar graphs, the conjecture has been verified by Borodin \cite{borodin} for $\Delta \geq 9$, by Yap \cite{yap} for $\Delta \geq 8$, and by Sanders and Zhao \cite{sanders} for $\Delta = 7$. Therefore, the only remaining open case for planar graphs is when $\Delta = 6$. In this case, some partial results are known: Sun et al.~\cite{sun} proved that every planar graph $G$ with maximum degree $6$ is totally $8$-colorable if no two triangles in $G$ share a common edge. In a recent paper, Zhu and Xu \cite{zhu} improved this result by proving that $\chi''(G) \leq 8$ when $\Delta = 6$ and $G$ does not contain any subgraph isomorphic to a $4$-fan (see Figure~\ref{fig:4-fan}).

It is known that $\chi''(G) \geq \Delta + 1$, since any two adjacent or incident elements must receive different colors. In 1989, Sanchez-Arroyo \cite{sanchez} proved that, in general, it is NP-complete to decide whether $\chi''(G) = \Delta + 1$. For planar graphs, it has been shown that $\chi''(G) = \Delta + 1$ when $\Delta \geq 9$ \cite{borodin,kowalik,wang2007}. Thus, it remains an open question whether every planar graph with maximum degree $\Delta \in \{4, 5, 6, 7, 8\}$ is totally $(\Delta + 1)$-colorable.
Note that $\chi''(G) = \Delta + 2=5$ when $G$ is isomorphic to the complete graph with  four vertices.

Some recent papers are devoted to proving that $\chi''(G)=9$ when $\Delta=8$, under additional structural restrictions. For instance, Hou et al.~\cite{hou} proved that $\chi''(G)=9$ when $\Delta = 8$ and $G$ contains no $5$- or $6$-cycles. Later, Roussel and Zhu \cite{rous} showed that $\chi''(G) = 9$ when $\Delta = 8$ and, for each vertex $v$, there exists an integer $k_v \in \{3, 4, 5, 6, 7, 8\}$ such that no $k_v$-cycle contains $v$. On the other hand, Wang et al.~\cite{wang2017} proved that $\chi''(G)= 9$ when $\Delta =8$ and $G$ contains no adjacent $p,q$-cycles for some $p, q \in \{3, 4, 5, 6, 7\}$. More recently, Wang et al. \cite{wang2023} generalized this result and showed that $\chi''(G)=9$ when $\Delta =8$ and $G$ has no adjacent $p,q$-cycles for some $p, q \in\{3, 4, 5, 6, 7, 8\}$.

In this paper, we provide a sufficient condition for planar graphs with maximum degree eight to be totally $9$-colorable.

\begin{theorem}\label{thm:main}
Let $G$ be a planar graph of maximum degree eight. If $G$ does not contain any subgraph isomorphic to a $4$-fan (see Figure~\ref{fig:4-fan}), then $G$ has a total $9$-coloring.
\end{theorem}

\begin{figure}[htb]
\centering   
\begin{tikzpicture}[scale=1]
\node [nod2w] at (0,0) (v)  {};
\node [nod2w] at (-1.5,-.25) (v5)   {}
	edge [] (v);
\node [nod2w] at (-1.25,1) (v6)  {}
	edge [] (v)
	edge [] (v5);	
\node [nod2w] at (1.25,1) (v1)  {}
	edge [] (v);
\node [nod2w] at (1.5,-.25) (v2)  {}
	edge [] (v)
	edge [] (v1);
\node [nod2w] at (0,1.5) (v7)   {}
	edge [] (v)
	edge [] (v6)
	edge [] (v1);					
\end{tikzpicture} 
\caption{A $4$-fan}
\label{fig:4-fan}
\end{figure}

Given a planar  graph $G$, we denote by $\ell(f)$ the length of a face $f$, and by $d(v)$ the degree of a vertex $v$. A \emph{$k$-vertex} is a vertex of degree $k$. A \emph{$k^{-}$-vertex} is a vertex of degree at most $k$ while a \emph{$k^{+}$-vertex} is a vertex of degree at least $k$. The notions of \emph{$k$-face}, \emph{$k^-$-face}, and \emph{$k^+$-face} are defined analogously. A vertex $u\in N(v)$ is called \emph{$k$-neighbour} (resp.~\emph{$k^-$-neighbour}, \emph{$k^+$-neighbour}) of $v$ if $d(u)=k$ (resp.~$d(u)\leq k$, $d(u)\geq k$). 

For a vertex $v\in V(G)$, we use $m_k(v)$ to denote the number of $k$-faces incident with $v$, and $n_k(v)$ to denote  the number of $k$-vertices adjacent to $v$. A cycle of length $3$ is called a \emph{triangle}. A \emph{$(p,q,r)$-triangle} is a $3$-face whose boundary is formed by vertices of degrees $p,q,r$. Two faces $f_1$ and $f_2$ are said to be adjacent if they share a common edge. A set of faces $f_1,f_2,\ldots,f_k$ around a vertex $v$ are called \emph{consecutive} if each pair of $f_i, f_{i+1}$, for $1\leq i\leq k-1 $, are  adjacent.

\section{The Proof of Theorem \ref{thm:main}} \label{sec:premECE}

\subsection{The Structure of Minimum Counterexample} \label{sub:premECE}~~\medskip

Let  $G$ be a minimal counterexample to Theorem \ref{thm:main} with minimum $|V(G)\cup E(G)|$. So, $G$  does not contain any subgraph isomorphic to a $4$-fan, and $G$ does not admit any total $9$-coloring, but for any $x\in V(G)\cup E(G)$, the graph $G-x$ admits a total $9$-coloring. Obviously,  $G$ is a connected graph.

\begin{lemma}\label{lem:min-deg}
$\delta(G)\geq 2$. 
\end{lemma}
\begin{proof}
Assume for a contradiction that $G$ has a vertex $u$ of degree $1$ with $N(u)=\{v\}$. By the minimality of $G$,  the graph $G-uv$ admits a proper $9$-coloring, and uncolor the vertex $u$. Since $d(u)+d(v)\leq \D+1=9$, and $u$ is uncolored vertex, we deduce that $uv$ has an available color, and so we color it. Later we color $u$ with an available color as it has two forbidden colors.
\end{proof}

Consider a $4^-$-vertex $u$ of $G$. If we have a proper $9$-coloring of $G$ in which $u$ is uncolored, then we can easily extend this coloring to the whole $G$ since $u$ has at most $8$ forbidden colors. Therefore, we assume that such vertices are colored at the end, as stated below.

\begin{remark}\label{rem:4-vert-later}
All $4^-$-vertices are colored at the end, since those vertices have always an available color.
\end{remark}

\begin{lemma}\label{lem:uv-10}
If $uv$ is an edge, and $u$ is a $4^-$-vertex, then $d(u)+d(v)\geq 10$.
\end{lemma}
\begin{proof}
Suppose that $uv$ is an edge such that $u$ is a $4^-$-vertex and $d(u)+d(v)\leq 9$. By the minimality of  $G$, the graph $G-uv$ has a proper $9$-coloring in which $u$ is uncolored. Note that it  suffices to color only the edge $uv$ by Remark \ref{rem:4-vert-later}. Since $d(u)+d(v)\leq 9$ and $u$ is uncolored, there is at least one available color for $uv$, and so we color it.
\end{proof}

\begin{lemma}\label{lem:8-has-one-2}
An $8$-vertex is adjacent to at most one $2$-vertex.
\end{lemma}
\begin{proof}
Let $v$ be an $8$-vertex. Assume to the contrary that $v$ has two $2$-neighbours $x_1,x_2$. Denote by $y_i$ the neighbour of $x_i$ other than $v$, for $i\in\{1,2\}$. 

First suppose that $vy_1,vy_2\notin E(G)$. If $y_1=y_2$, then we remove $x_1,x_2$ and add $vy_1$. Let $G'$ be the resulting graph. By the minimality of $G$, the graph $G'$ has a proper $9$-coloring $\varphi$. 
 Now we modify the coloring $\varphi$ with respect to $G$. Let us give $\varphi(vy_1)$ to each of $vx_1,y_1x_2$. Note that each of $vx_2$ and $x_1y_1$ has an available color, and so we color them. Also, we color $x_1 $ and $x_2$ by Remark \ref{rem:4-vert-later}. Observe that the resulting coloring is a proper $9$-coloring of $G$.
If $y_1\neq y_2$,  then we remove $x_1,x_2$ and add $vy_1,vy_2$. Let $G'$ be the resulting graph. By the minimality, $G'$ has a proper $9$-coloring $\varphi$. Now we modify the coloring $\varphi$ with respect to $G$.  Let us give $\varphi(vy_1)$ to each of $vx_2,x_1y_1$, and $\varphi(vy_2)$ to each of $vx_1,x_2y_2$, and color $x_1 $ and $x_2$ by Remark \ref{rem:4-vert-later}.  Observe that the resulting coloring is a proper $9$-coloring of $G$.

Now we suppose that $v$ is adjacent to at least one of $y_1,y_2$, say  $vy_1\in E(G)$. Consider a proper $9$-coloring $\varphi$ of $G-vx_2$. Suppose that $vx_2$ has no available color. This means that none of the edges incident to $v$ is colored with $\varphi(x_2y_2)$; otherwise there would exist an available color for $vx_2$.  If $\varphi(x_1y_1)=\varphi(x_2y_2)$, then we interchange the colors of $x_1y_1$ and $vy_1$, and later we color $vx_2$ with $\varphi(vy_1)$.  If $\varphi(x_1y_1)\neq \varphi(x_2y_2)$,  then we recolor $vx_1$ with $\varphi(x_2y_2)$, and we color $vx_2$ with $\varphi(vx_1)$.
\end{proof}

In what follows, we will show that certain configurations are \emph{reducible}, meaning they cannot appear as subgraphs of $G$. These configurations are illustrated in the figures. In each figure, black bullets represent vertices whose neighbours are exactly as drawn, while white bullets represent vertices that may have additional neighbours beyond those shown.

\begin{lemma}\label{lem:7-has-no-two-3-on-3face}
Let $v$ be a $7$-vertex. If $v$ is adjacent to a $3$-vertex on a $3$-face, then $v$  is not adjacent to any  other $3$-vertex on a $3$-face.
\end{lemma}
\begin{proof}
Let $v$ be adjacent to a $3$-vertex $u$ on a $3$-face $f_1$, and assume to the contrary that $v$  is adjacent to another $3$-vertex, say $w$, on a $3$-face $f_2$. By Lemma \ref{lem:uv-10}, we have $f_1\neq f_2$.
Consider a proper $9$-coloring of $G-uv$, and uncolor the vertices $u$ and $w$. We will find an appropriate color only for $uv$ by Remark \ref{rem:4-vert-later}.

First suppose that $f_1$ and $f_2$ are adjacent. We may suppose that the coloring is the one shown in Figure \ref{fig:conf7-384}. 
If $9\notin \{a,b\}$ then we color $vw$ with $9$, and $uv$ with $5$.  Similarly, if $b\neq 8$, then we color $vw$ with $8$, and $uv$ with $5$ where we note that $a\neq 8$ as the color of $ux$ is $8$. Therefore $a=9$ and $b=8$. Let us now interchange the colors of $xv$ and $xu$, and later we color $vw$ with $6$, and $uv$ with $5$.

Next suppose that $f_1$ and $f_2$ are not adjacent. We may suppose that the coloring is the one shown in Figure \ref{fig:conf7-38-38}. 
If $9\notin \{a,b\}$, then we color $vw$ with $9$, and $uv$ with $5$. Similarly, if $8\notin \{a,b\}$, then we color $vw$ with $8$, and $uv$ with $5$. Thus $\{a,b\}=\{8,9\}$. Let us now interchange the colors of $xv$ and $xu$, and later we color $vw$ with $6$, and $uv$ with $5$.
\end{proof}

\begin{figure}[htb]
\centering  
\subfigure[]{ 
\label{fig:conf7-384}
\begin{tikzpicture}[scale=1]
\node [nod2] at (0,0) (v) [label={[xshift=-.13cm, yshift=-0.05cm]{\scriptsize $v$}}] {};
\node  at (0,0) (v) [label=below:{\scriptsize $7$}] {};	

\node [nod2] at (.75,-1.25) (w) [label=below:{\scriptsize $w$}] {}
	edge node[midway,right,font=\scriptsize] {$5$} (v);		
\node [nod2w] at (-.75,-1.25) (v4) {}
	edge [] node[midway,left,font=\scriptsize] {$4$} (v);
\node [nod2w] at (-1.5,-.25) (v5)   {}
	edge [] node[midway,above,font=\scriptsize] {$3$} (v);
\node [nod2w] at (-1.25,1) (v6)  {}
	edge [] node[midway,above,font=\scriptsize] {$2$} (v);	
\node [nod2] at (1.25,1) (u) [label=right:{\scriptsize $u$}] {}
	edge  [green!70!black]  (v);
\node [nod2w] at (1.5,-.25) (v2)[label=right:{\scriptsize $x$}]  {}
	edge [] node[midway,above,font=\scriptsize] {$6$} (v)
	edge [] node[midway,right,font=\scriptsize] {$8$} (u)
	edge [] node[midway,right,font=\scriptsize] {$a$} (w);
\node [nod2w] at (0,1.5) (v7)   {}
	edge [] node[midway,right,font=\scriptsize] {$1$} (v);
	
\node [nod2w] at (1,2) (v8)   {}
	edge [] node[midway,right,font=\scriptsize] {$9$} (u);

\node [nod2w] at (1.5,-2) (v10)   {}
	edge [] node[midway,right,font=\scriptsize] {$b$} (w);	
						
\end{tikzpicture}}
\subfigure[]{
\label{fig:conf7-38-38}
\begin{tikzpicture}[scale=1]
\node [nod2] at (0,0) (v) [label={[xshift=-.13cm, yshift=-0.05cm]{\scriptsize $v$}}] {};
\node  at (0,0) (v) [label=below:{\scriptsize $7$}] {};	

\node [nod2] at (.75,-1.25) (w) [label=below:{\scriptsize $w$}] {}
	edge node[midway,right,font=\scriptsize] {$5$} (v);		
\node [nod2w] at (-.75,-1.25) (v4) [label=below:{\scriptsize $y$}] {}
	edge [] node[midway,left,font=\scriptsize] {$4$} (v)
	edge [] node[midway,below,font=\scriptsize] {$b$} (w);
\node [nod2w] at (-1.5,-.25) (v5)   {}
	edge [] node[midway,above,font=\scriptsize] {$3$} (v);
\node [nod2w] at (-1.25,1) (v6)  {}
	edge [] node[midway,above,font=\scriptsize] {$2$} (v);	
\node [nod2] at (1.25,1) (u) [label=right:{\scriptsize $u$}] {}
	edge  [green!70!black]  (v);
\node [nod2w] at (1.5,-.25) (v2)[label=right:{\scriptsize $x$}]  {}
	edge [] node[midway,above,font=\scriptsize] {$6$} (v)
	edge [] node[midway,right,font=\scriptsize] {$8$} (u);
\node [nod2w] at (0,1.5) (v7)   {}
	edge [] node[midway,right,font=\scriptsize] {$1$} (v);
	
\node [nod2w] at (1,2) (v8)   {}
	edge [] node[midway,right,font=\scriptsize] {$9$} (u);

\node [nod2w] at (2,-1.75) (v10)   {}
	edge [] node[midway,below,font=\scriptsize] {$a$} (w);							
\end{tikzpicture}}
\subfigure[]{
\label{fig:conf8-283}
\begin{tikzpicture}[scale=1]
\node [nod2] at (0,0) (v) [label={[xshift=.35cm, yshift=-0.67cm]{\scriptsize $8$}}] {};
\node at (0,0) (v) [label=above:{\scriptsize $v$}] {};
\node [nod2] at (1.5,.65) (v1) [label=above:{\scriptsize $u$}] {}
	edge [green!70!black] (v);
\node [nod2w] at (1.5,-.65) (v2) [label=right:{\scriptsize $x$}] {}
	edge [] node[midway,below,font=\scriptsize] {$7$} (v)
	edge [] node[midway,right,font=\scriptsize] {$9$} (v1);	
\node [nod2] at (.65,-1.5) (v3) [label=below:{\scriptsize $w$}] {}
	edge  node[midway,left,font=\scriptsize] {$6$} (v)
	edge [] node[midway,below,font=\scriptsize] {$b$} (v2);		
\node [nod2w] at (-.65,-1.5) (v4) [label=below:{\scriptsize $y$}]  {}
	edge [] node[midway,left,font=\scriptsize] {$5$} (v);
\node [nod2w] at (-1.5,-.65) (v5)   {}
	edge []  node[midway,above,font=\scriptsize] {$4$}(v);
\node [nod2w] at (-1.5,.65) (v6)  {}
	edge [] node[midway,above,font=\scriptsize] {$3$} (v);	

\node [nod2w] at (-.65,1.5) (v7)   {}
	edge [] node[midway,right,font=\scriptsize] {$2$} (v);	
\node [nod2w] at (.65,1.5) (v8)   {}
	edge [] node[midway,right,font=\scriptsize] {$1$} (v);	
	
\node [nod2w] at (1.85,-2) (v9)   {}
	edge [] node[midway,below,font=\scriptsize] {$a$} (v3);					
\end{tikzpicture}}
\subfigure[]{
\label{fig:conf8-28-38}
\begin{tikzpicture}[scale=1]
\node [nod2] at (0,0) (v) [label={[xshift=.35cm, yshift=-0.67cm]{\scriptsize $8$}}] {};
\node at (0,0) (v) [label=above:{\scriptsize $v$}] {};
\node [nod2] at (1.5,.65) (v1) [label=above:{\scriptsize $u$}] {}
	edge [green!70!black] (v);
\node [nod2w] at (1.5,-.65) (v2) [label=right:{\scriptsize $x$}] {}
	edge [] node[midway,below,font=\scriptsize] {$7$} (v)
	edge [] node[midway,right,font=\scriptsize] {$9$} (v1);	
\node [nod2] at (.65,-1.5) (v3) [label=below:{\scriptsize $w$}] {}
	edge  node[midway,left,font=\scriptsize] {$6$} (v);		
\node [nod2w] at (-.65,-1.5) (v4) [label=below:{\scriptsize $y$}]  {}
	edge [] node[midway,left,font=\scriptsize] {$5$} (v)
	edge [] node[midway,below,font=\scriptsize] {$b$} (v3);
\node [nod2w] at (-1.5,-.65) (v5)   {}
	edge []  node[midway,above,font=\scriptsize] {$4$}(v);
\node [nod2w] at (-1.5,.65) (v6)  {}
	edge [] node[midway,above,font=\scriptsize] {$3$} (v);	

\node [nod2w] at (-.65,1.5) (v7)   {}
	edge [] node[midway,right,font=\scriptsize] {$2$} (v);	
\node [nod2w] at (.65,1.5) (v8)   {}
	edge [] node[midway,right,font=\scriptsize] {$1$} (v);	
	
\node [nod2w] at (1.85,-2) (v9)   {}
	edge [] node[midway,below,font=\scriptsize] {$a$} (v3);					
\end{tikzpicture}}
\caption{ }
\end{figure}

\begin{lemma}\label{lem:8-no-2-on3face-and-3-on3face}
Let $v$ be an $8$-vertex. If $v$ is adjacent to a $2$-vertex on a $3$-face, then $v$  is not adjacent to any $3$-vertex on a $3$-face.
\end{lemma}
\begin{proof}
Let $v$ be adjacent to a $2$-vertex $u$ on a $3$-face $f_1$, and assume to the contrary that $v$  is adjacent to a $3$-vertex $w$ on a $3$-face $f_2$. By Lemma \ref{lem:uv-10}, we have $f_1\neq f_2$. Consider a proper $9$-coloring of $G-uv$, and uncolor the vertices $u,w$. We will find an appropriate color only for $uv$ by Remark \ref{rem:4-vert-later}.

First suppose that $f_1$ and $f_2$ are adjacent. We may suppose that the coloring is the one shown in Figure \ref{fig:conf8-283}. Note that $a=9$; otherwise we would color $vw$ with $9$ and $uv$ with   $6$  where we note that $b\neq 9$ as the color of $ux$ is $9$. It then follows that we interchange the colors of $xv$ and $xu$, and we color $vw$ with $7$, $uv$ with $6$ where we note that $b\neq 7$ as the color of $vx$ is $7$.

Next suppose that $f_1$ and $f_2$ are not adjacent. We may suppose that the coloring is the one shown in Figure \ref{fig:conf8-28-38}. Note that
$9\in \{a,b\}$; otherwise we would color $vw$ with $9$ and $uv$ with   $6$. Moreover, $7\in \{a,b\}$; otherwise we would interchange the colors of $xv$ and $xu$, and we color $vw$ with $7$, $uv$ with $6$. Thus we conclude that $\{a,b\}=\{7,9\}$. Let us interchange the colors of $yv$ and $yw$. If $b=9$, then  we color $uv$ with $5$. If $b=7$, then we interchange the colors of $xv$ and $xu$, and we color $uv$ with $5$.
\end{proof}

\begin{lemma}\label{lem:8-v-has-no2-when-3-diamond}
Let $v$ be an $8$-vertex. If $v$ is adjacent to a $3$-vertex on two $3$-faces, then $v$  has no $2$-neighbour.
\end{lemma}
\begin{proof}

Assume that $v$ is adjacent to a $3$-vertex $w$ on two $3$-faces, and $v$  has a $2$-neighbour $u$. 
Consider a proper $9$-coloring of $G-uv$, and uncolor the vertices $u,w$. We will find an appropriate color  only for $uv$ by Remark \ref{rem:4-vert-later}.

We may suppose that the coloring is the one shown in Figure \ref{fig:conf8-2-838}. Notice first that 
$9\in \{a,b\}$; otherwise we would color $vw$ with $9$ and $uv$ with   $6$.
We may assume without loss of generality that $a=9$. If $b\neq 7$, then we interchange the colors of $xv$ and $xw$, and we color $uv$ with $7$. If $b=7$, then we interchange the colors of $xv$ and $xw$ as well as the colors of $yv$ and $yw$, and then  we color $uv$ with $5$.
\end{proof}

\begin{figure}[htb]
\centering  
\subfigure[]{ 
\label{fig:conf8-2-838}
\begin{tikzpicture}[scale=1]
\node [nod2] at (0,0) (v) [label={[xshift=.35cm, yshift=-0.65cm]{\scriptsize $8$}}] {};
\node at (0,0) (v) [label=above:{\scriptsize $v$}] {};
\node [nod2] at (1.5,.65) (v1) [label=above:{\scriptsize $u$}] {}
	edge [green!70!black] (v);

\node [nod2w] at (2.5,.65) (v11)   {}
	edge [] node[midway,below,font=\scriptsize] {$9$} (v1);		
	
\node [nod2w] at (1.5,-.65) (v2) [label=right:{\scriptsize $x$}] {}
	edge [] node[midway,below,font=\scriptsize] {$7$} (v);
\node [nod2] at (.65,-1.5) (v3) [label=below:{\scriptsize $w$}] {}
	edge  node[midway,left,font=\scriptsize] {$6$} (v)
	edge node[midway,below,font=\scriptsize] {$a$}  (v2);		
\node [nod2w] at (-.65,-1.5) (v4) [label=below:{\scriptsize $y$}]  {}
	edge [] node[midway,left,font=\scriptsize] {$5$} (v)
	edge [] node[midway,below,font=\scriptsize] {$b$} (v3);
\node [nod2w] at (-1.5,-.65) (v5)   {}
	edge []  node[midway,above,font=\scriptsize] {$4$}(v);
\node [nod2w] at (-1.5,.65) (v6)  {}
	edge [] node[midway,above,font=\scriptsize] {$3$} (v);	

\node [nod2w] at (-.65,1.5) (v7)   {}
	edge [] node[midway,right,font=\scriptsize] {$2$} (v);	
\node [nod2w] at (.65,1.5) (v8)   {}
	edge [] node[midway,right,font=\scriptsize] {$1$} (v);					
\end{tikzpicture}}
\subfigure[]{ 
\label{fig:conf8-8388-8383}
\begin{tikzpicture}[scale=1]
\node [nod2] at (0,0) (v) [label={[xshift=.35cm, yshift=-0.65cm]{\scriptsize $8$}}] {};
\node at (0,0) (v) [label=above:{\scriptsize $v$}] {};
\node [nod2] at (1.5,.65) (v1) [label=right:{\scriptsize $u$}] {}
	edge node[midway,below,font=\scriptsize]  {$6$} (v);

\node [nod2w] at (1.5,-.65) (v2) [label=right:{\scriptsize $y$}] {}
	edge [] node[midway,below,font=\scriptsize] {$7$} (v)
	edge [] node[midway,right,font=\scriptsize] {$b$} (v1);	
\node [nod2] at (.65,-1.5) (v3) [label=below:{\scriptsize $z$}] {}
	edge   [green!70!black](v)
	edge [] node[midway,left,font=\scriptsize] {$c$} (v2);		
\node [nod2w] at (-.65,-1.5) (v4) [label=below:{\scriptsize $s$}]  {}
	edge [] node[midway,left,font=\scriptsize] {$1$} (v);
\node [nod2] at (-1.5,-.65) (v5) [label=below:{\scriptsize $w$}]  {}
	edge []  node[midway,above,font=\scriptsize] {$2$}(v)
	edge []  node[midway,below,font=\scriptsize] {$e$}(v4);
\node [nod2w] at (-1.5,.65) (v6) [label=left:{\scriptsize $p$}] {}
	edge [] node[midway,above,font=\scriptsize] {$3$} (v)
	edge [] node[midway,left,font=\scriptsize] {$f$} (v5);	

\node [nod2w] at (-.65,1.5) (v7) [label=above:{\scriptsize $t$}]  {}
	edge [] node[midway,right,font=\scriptsize] {$4$} (v)
	edge [] node[midway,left,font=\scriptsize] {$g$} (v6);	
\node [nod2w] at (.65,1.5) (v8)  [label=above:{\scriptsize $x$}] {}
	edge [] node[midway,right,font=\scriptsize] {$5$} (v)
	edge [] node[midway,right,font=\scriptsize] {$a$} (v1);	
\node [nod2w] at (1.7,-2) (v9)  [label=right:{\scriptsize $r$}] {}
	edge [] node[midway,below,font=\scriptsize] {$d$} (v3);								
\end{tikzpicture}}
\caption{ }
\end{figure}

\begin{lemma}\label{lem:8-v-has-no-two-3-on-3-faces}
Let $v$ be an $8$-vertex. Suppose that $v$ is incident to a pair of three consecutive $3$-faces. If there exist two $3$-neighbours $u,w$ of $v$ such that each of them is adjacent to $v$ on two $3$-faces, then $v$  is not adjacent to any $3$-vertex $z$ with $z\notin \{u,w\}$. 
\end{lemma}
\begin{proof}

Suppose that $v$ is adjacent to a $3$-vertex $u$ on two $3$-faces, and a $3$-vertex $w$ on two $3$-faces. Assume for a contradiction that $v$  is adjacent to a $3$-vertex $z$ with $z\notin \{u,w\}$.  
Consider a proper $9$-coloring of $G-vz$, and uncolor the vertices $u,w,z$. We will find an appropriate color only for $vz$ by Remark \ref{rem:4-vert-later}.

Using the facts that $G$ has no $4$-fan, that $z$ is not adjacent to $u$ or $w$ by Lemma \ref{lem:uv-10} and that $v$ is incident to a pair of three consecutive $3$-faces, we deduce that $z$ has two common neighbours with $u$ or $w$. By symmetry, we assume that  $z$ has two common neighbours with $u$. Thus, 
we may suppose that the coloring is the one shown in Figure \ref{fig:conf8-8388-8383}. Observe that $9\in \{c,d\}$; otherwise $vz$  would have an available color.  

Suppose first that $c=9$. If $d\in \{1,2,3,4\}$ then  $9\in \{a,b\}$; otherwise we would color $vu$ with $9$ and $vz$ with $6$. It follows that $a=9$ as $c=9$. Therefore we interchange the colors of $xu$ and $xv$, and we color $vz$ with $5$. 
If $d\in \{5,6,7,8\}$ then $9\in \{e,f\}$; otherwise we would color $vw$ with $9$ and $vz$ with $2$. It then follows that $f=9$; otherwise we would interchange the colors of $sw$ and $sv$, and we color $vz$ with $1$. Therefore we interchange the colors of $pv$ and $pw$, and we color $vz$ with $3$. 

Suppose next that $d=9$. 
If $c\in \{1,2,3,4\}$ then $9\in \{a,b\}$; otherwise we would color $uv$ with $9$ and $vz$ with $6$. It follows that  $b=9$; otherwise we would interchange the colors of $xv$ and $xu$, and we color $vz$ with $5$. Thus, we interchange the colors of $yv$ and $yu$, and we color $vz$ with $7$.
If $c\in \{5,6,7,8\}$ then $9\in \{e,f\}$; otherwise we would color $vw$ with $9$ and $vz$ with $2$. It follows that  $f=9$; otherwise  we would interchange the colors of $sw$ and $sv$, and we color $vz$ with $1$. Therefore we interchange the colors of $pv$ and $pw$, and we color $vz$ with $3$. 
\end{proof}

\begin{figure}[htb]
\centering   
\begin{tikzpicture}[scale=1]
\node [nod2] at (0,0) (v) [label={[xshift=.35cm, yshift=-0.65cm]{\scriptsize $8$}}] {};
\node at (0,0) (v) [label=above:{\scriptsize $v$}] {};
\node [nod2w] at (1.5,.65) (v1) [label=right:{\scriptsize $v_2$}] {}
	edge [] node[midway,below,font=\scriptsize] {$3$}(v);

\node [nod2] at (1.5,-.65) (v2) [label=right:{\scriptsize $v_3$}] {}
	edge [] node[midway,below,font=\scriptsize] {$4$} (v)
	edge [] node[midway,right,font=\scriptsize] {$c_3$} (v1);	
\node [nod2] at (.65,-1.5) (v3) [label={[xshift=.1cm, yshift=-0.6cm]{\scriptsize $v_4$}}] {}
	edge node[midway,left,font=\scriptsize] {$5$}  (v);		
\node [nod2] at (-.65,-1.5) (v4)   {}
	edge  [] node[midway,left,font=\scriptsize] {$6$} (v);
\node [nod2] at (-1.5,-.65) (v5) [label=left:{\scriptsize $v_{i\text{-}1}$}]  {}
	edge [] node[midway,above,font=\scriptsize] {$7$} (v);
\node [nod2] at (-1.5,.65) (v6) [label=above:{\scriptsize $v_i$}] {}
	edge [green!70!black] (v);	

\node [nod2w] at (-.65,1.5) (v7)   {}
	edge [] node[midway,right,font=\scriptsize] {$1$} (v);	
\node [nod2] at (.65,1.5) (v8) [label=left:{\scriptsize $v_1$}] {}
	edge [] node[midway,right,font=\scriptsize] {$2$} (v)
	edge [] node[midway,right,font=\scriptsize] {$c_2$} (v1);		
	
\node [nod2w] at (0,-2.7) (v9) [label=below:{\scriptsize $p_2$}]  {}
	edge [] node[midway,below,font=\scriptsize] { \  $\ c_6$} (v3)
	edge [dotted]  (v4);		

\node [nod2w] at (2,-2) (v9) [label=right:{\scriptsize $p_1$}]  {}
	edge [] node[midway,below,font=\scriptsize] {$c_5$} (v3)	
	edge [] node[midway,right,font=\scriptsize] {$c_4$} (v2);		

\node [nod2w] at (-2,-2) (v9) [label=below:{\scriptsize $p_{i\text{-}4}$}]  {}
	edge [dotted]  (v4)
	edge [] node[midway,left,font=\scriptsize] {$c_{i}$} (v5);

\node [nod2w] at (-3.3,0) (v9) [label=left:{\scriptsize $p_{i\text{-}3}$}]  {}
	edge [] node[midway,left,font=\scriptsize] {$c_{i+1}$} (v5)
	edge [] node[midway,above,font=\scriptsize] {$9$} (v6);
		
\node [nod2w] at (1,2.5) (v9)  {}
	edge [] node[midway,left,font=\scriptsize] {$c_1$} (v8);						
\end{tikzpicture}
\caption{ }
\label{fig:conf8-383--3--3--3--2}
\end{figure}

\begin{lemma}\label{lem:8-has-no-233383}
Let $v$ be an $8$-vertex and denote by $v_1,v_2\ldots,v_8$ its neighbours with a cyclic orientation such that exactly one $v_i$, for some $4\leq i \leq 7$, is a $2$-vertex. Suppose that $v_1$ (resp.~$v_3$) is contained in a $3$-face $f_1$ (resp.~$f_2$) such that $f_1$ and $f_2$ are adjacent.  If  $v_1,v_{3},v_4,\ldots,v_{i-1}$ are $3$-vertices, then at least one of the edges $vv_j$, for some $3\leq j \leq i-1$, is contained in a $5^+$-face.
\end{lemma}
\begin{proof}
Suppose that  $v_1$ (resp.~$v_3$) is contained in a $3$-face $f_1$ (resp.~$f_2$) such that $f_1$ and $f_2$ are adjacent. Let $v_i$, for $4\leq i \leq 7$, be a $2$-vertex, and let the vertices $v_4,v_{5},\ldots,v_{i-1}$ be $3$-vertices. Assume to the contrary that each edge $vv_j$, for $3\leq j \leq i-1$, is contained in a $4$-face, where we note that those faces cannot be $3$-faces by Lemma \ref{lem:uv-10}.

Consider a proper $9$-coloring of $G-vv_i$, and uncolor the vertices $v_1,v_3,v_4,\ldots, v_i$. We will find an appropriate color only for $vv_i$ by Remark \ref{rem:4-vert-later}.
We may suppose that the coloring is the one shown in Figure \ref{fig:conf8-383--3--3--3--2}. Note that $9\in\{c_1,c_2\}$; otherwise we would color $vv_1$ with $9$ and $vv_i$ with $2$. Similarly,  $9\in\{c_2,c_3\}$, $9\in\{c_3,c_4\}$, $9\in\{c_5,c_6\}$, and  $9\in\{c_i,c_{i+1}\}$. This implies that $c_1=c_3=c_5=\ldots=c_i=9$ as the color of $v_ip_{i-3}$ is $9$.
If $c_4\neq 3$ then we  interchange the colors of $v_2v$ and $v_2v_3$, and we color $vv_i$ with $3$. Thus we may assume that $c_4=3$. If $c_6\neq 3$ then we  interchange the colors of $v_2v$ and $v_2v_3$ as well as  the colors of $p_1v_3$ and $p_1v_4$, and we color $vv_i$ with $3$. Hence $c_6=3$. 
Continuing this process, we eventually obtain that $c_{i+1}=3$.
Now we can alternate the color of the edges of the path $vv_2v_3p_1v_4p_2\ldots v_{i-1}p_{i-3}v_i$. We color $vv_1$ with $3$ and $vv_i$ with $2$, where we note that $c_2\neq 3$ as the color of $vv_2$ is $3$.
\end{proof}

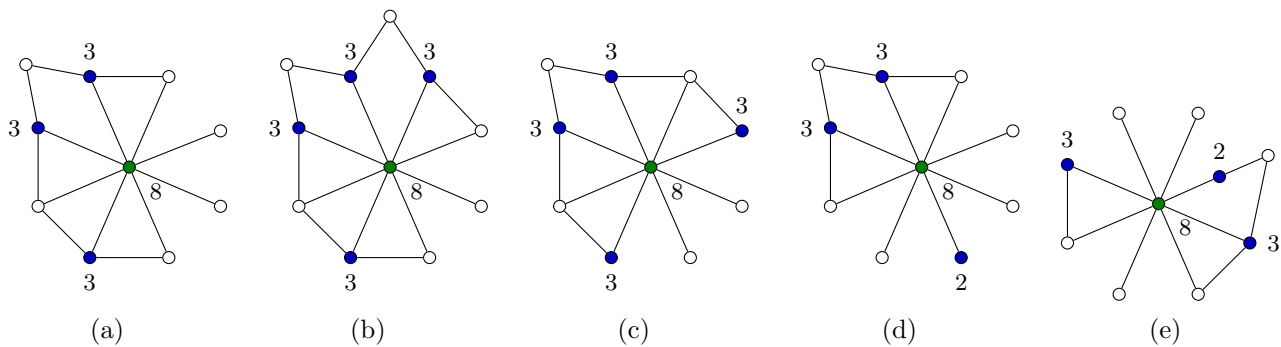
\begin{figure}[htb]
\centering  
\subfigure[]{
\label{fig:8-8383--38}
\begin{tikzpicture}[scale=.8]
\node [nodr] at (0,0) (v) [label={[xshift=.35cm, yshift=-0.65cm]{\scriptsize $8$}}] {};
\node [nod2w] at (1.5,.6) (v1)  {}
	edge [] (v);
\node [nod2w] at (1.5,-.65) (v2) {}
	edge [] (v);
\node [nod2w] at (.65,-1.5) (v3)  {}
	edge  (v);		
\node [nod2] at (-.65,-1.5) (v4)  [label=below:{\scriptsize $3$}] {}
	edge [] (v)
	edge [] (v3);
\node [nod2w] at (-1.5,-.65) (v5)   {}
	edge [] (v)
	edge [] (v4);
\node [nod2] at (-1.5,.65) (v6)  [label=left:{\scriptsize $3$}] {}
	edge [] (v)
	edge [] (v5);	
\node [nod2] at (-.65,1.5) (v7)  [label=above:{\scriptsize $3$}] {}
	edge [] (v);	
\node [nod2w] at (.65,1.5) (v8)    {}
	edge [] (v)	
	edge [] (v7);
\node [nod2w] at (-1.7,1.7) (v11)   {}
	edge [] (v6)
	edge [] (v7);											
\end{tikzpicture} }
\subfigure[]{
\label{fig:8-8383--3--38}
\begin{tikzpicture}[scale=.8]
\node [nodr] at (0,0) (v) [label={[xshift=.35cm, yshift=-0.65cm]{\scriptsize $8$}}] {};
\node [nod2w] at (1.5,.6) (v1)  {}
	edge [] (v);
\node [nod2w] at (1.5,-.65) (v2) {}
	edge [] (v);
\node [nod2w] at (.65,-1.5) (v3)  {}
	edge  (v);		
\node [nod2] at (-.65,-1.5) (v4)  [label=below:{\scriptsize $3$}] {}
	edge [] (v)
	edge [] (v3);
\node [nod2w] at (-1.5,-.65) (v5)   {}
	edge [] (v)
	edge [] (v4);
\node [nod2] at (-1.5,.65) (v6)  [label=left:{\scriptsize $3$}] {}
	edge [] (v)
	edge [] (v5);	
\node [nod2] at (-.65,1.5) (v7)  [label=above:{\scriptsize $3$}] {}
	edge [] (v);	
\node [nod2] at (.65,1.5) (v8)  [label=above:{\scriptsize $3$}]   {}
	edge [] (v)	
	edge [] (v1);
\node [nod2w] at (-1.7,1.7) (v11)   {}
	edge [] (v6)
	edge [] (v7);	
\node [nod2w] at (0,2.5) (v12)    {}
	edge [] (v8)	
	edge [] (v7);											
\end{tikzpicture} }
\subfigure[]{
\label{fig:8-383--383}
\begin{tikzpicture}[scale=.8]
\node [nodr] at (0,0) (v) [label={[xshift=.35cm, yshift=-0.65cm]{\scriptsize $8$}}] {};
\node [nod2] at (1.5,.6) (v1)  [label=above:{\scriptsize $3$}] {}
	edge [] (v);
\node [nod2w] at (1.5,-.65) (v2) {}
	edge [] (v);
\node [nod2w] at (.65,-1.5) (v3)  {}
	edge  (v);		
\node [nod2] at (-.65,-1.5) (v4)  [label=below:{\scriptsize $3$}] {}
	edge [] (v);
\node [nod2w] at (-1.5,-.65) (v5)   {}
	edge [] (v)
	edge [] (v4);
\node [nod2] at (-1.5,.65) (v6)  [label=left:{\scriptsize $3$}] {}
	edge [] (v)
	edge [] (v5);	
\node [nod2] at (-.65,1.5) (v7)  [label=above:{\scriptsize $3$}] {}
	edge [] (v);	
\node [nod2w] at (.65,1.5) (v8)    {}
	edge [] (v)	
	edge [] (v7)
	edge [] (v1);
\node [nod2w] at (-1.7,1.7) (v11)   {}
	edge [] (v6)
	edge [] (v7);											
\end{tikzpicture} }
\subfigure[]{
\label{fig:8-83--38--2}
\begin{tikzpicture}[scale=.8]
\node [nodr] at (0,0) (v) [label={[xshift=.35cm, yshift=-0.65cm]{\scriptsize $8$}}] {};
\node [nod2w] at (1.5,.6) (v1)   {}
	edge [] (v);
\node [nod2w] at (1.5,-.65) (v2) {}
	edge [] (v);
\node [nod2] at (.65,-1.5) (v3) [label=below:{\scriptsize $2$}]  {}
	edge  (v);		
\node [nod2w] at (-.65,-1.5) (v4)  {}
	edge [] (v);
\node [nod2w] at (-1.5,-.65) (v5)   {}
	edge [] (v);
\node [nod2] at (-1.5,.65) (v6)  [label=left:{\scriptsize $3$}] {}
	edge [] (v)
	edge [] (v5);	
\node [nod2] at (-.65,1.5) (v7)  [label=above:{\scriptsize $3$}] {}
	edge [] (v);	
\node [nod2w] at (.65,1.5) (v8)    {}
	edge [] (v)	
	edge [] (v7);
\node [nod2w] at (-1.7,1.7) (v11)   {}
	edge [] (v6)
	edge [] (v7);											
\end{tikzpicture} }
\subfigure[]{
\label{fig:8-83-2--38}
\begin{tikzpicture}[scale=.8]
\node [nodr] at (0,0) (v) [label={[xshift=.35cm, yshift=-0.65cm]{\scriptsize $8$}}] {};
\node [nod2] at (1,.45) (v1)  [label=above:{\scriptsize $2$}] {}
	edge [] (v);
\node [nod2] at (1.5,-.65) (v2)[label=right:{\scriptsize $3$}] {}
	edge [] (v);
\node [nod2w] at (.65,-1.5) (v3)  {}
	edge  (v)
	edge  (v2);		
\node [nod2w] at (-.65,-1.5) (v4)  {}
	edge [] (v);
\node [nod2w] at (-1.5,-.65) (v5)   {}
	edge [] (v);
\node [nod2] at (-1.5,.65) (v6)  [label=above:{\scriptsize $3$}]  {}
	edge [] (v)
	edge [] (v5);	
\node [nod2w] at (-.65,1.5) (v7)   {}
	edge [] (v);	
\node [nod2w] at (.65,1.5) (v8)   {}
	edge [] (v);
\node [nod2w] at (1.8,.8) (v9)   {}
	edge [] (v1)
	edge [] (v2);						
\end{tikzpicture} }
\caption{Reducible configurations where the vertices, which have only one incident edge drawn, are allowed to permute.}
\label{fig:reducible-conf}
\end{figure}

In the graph $G$, we will show that the configurations in Figure \ref{fig:reducible-conf} are reducible, where the labels in those configurations indicate the degrees of vertices, and the unlabelled vertices (white bullets) can have any degree at least as large as shown in the figures.

\begin{lemma}\label{lem:8-8383--38}
The configuration in Figure \ref{fig:8-8383--38}  is reducible.
\end{lemma}
\begin{proof}

Suppose that $G$ contains this configuration. Denote by $v$ the $8$-vertex and by $t,y,z$ its $3$-neighbours.
Consider a proper $9$-coloring of $G-vt$, and uncolor the vertices $t,y,z$. We will find an appropriate color only for $vt$ by Remark \ref{rem:4-vert-later}.

We may suppose that the coloring is the one shown in Figure \ref{fig:conf8-8383--38}. Note that $9\in\{a,b\}$; otherwise $vt$ would have an available color. \medskip

\textbf{Case 1.} $a=9$.\medskip

$\quad$\textbf{Case 1.1.} $b=1$.

Then  $9\in\{c,d\}$; otherwise we would color $vy$ with $9$ and $vt$ with $6$.
Similarly we obtain $9\in\{e,f\}$. 

We may suppose that $d=7$; otherwise we would interchange   the colors of $xv$ and $xt$ as well as the colors of $wv$ and $wt$, and later we color $vy$ with $7$ and $vt$ with $6$. Moreover, $f=7$; otherwise we would interchange   the colors of $xv$ and $xt$ as well as the colors of $wv$ and $wt$, and later we color $vz$ with $7$ and $vt$ with $5$, where we note that $e\neq 7$ as $d=7$. It follows that $e=9$ as $9\in\{e,f\}$. Let us now interchange  the colors of $xv$ and $xt$, the colors of $wv$ and $wt$, and the colors of $uv$ and $uz$. Then we can color $vt$ with $4$.\medskip

$\quad$\textbf{Case 1.2.} $b=6$.

Observe that  $9\in\{e,f\}$; otherwise we would color $vz$ with $9$ and $vt$ with $5$.
Moreover  $1\in\{e,f\}$; otherwise we would interchange  the colors of $xv$ and $xt$, and we color $vz$ with $1$ and $vt$ with $5$. Thus $\{e,f\}=\{1,9\}$. 
If $f=9$, then we interchange  the colors of $uv$ and $uz$, and color $vt$ with $4$. Therefore $f=1$. In this case, we interchange the colors of $uv$ and $uz$ as well as the colors of $xv$ and $xt$, and we color $vt$ with $4$.\medskip

$\quad$\textbf{Case 1.3.} $b\notin \{1,6\}$.

Then  $9\in\{c,d\}$; otherwise we would color $vy$ with $9$ and $vt$ with $6$.
Moreover  $1\in\{c,d\}$; otherwise we would interchange  the colors of $xv$ and $xt$, and we color $vy$ with $1$ and $vt$ with $6$. Thus $\{c,d\}=\{1,9\}$. 
If $c=9$, then we interchange  the colors of $wv$ and $wy$, and color $vt$ with $7$. Therefore $c=1$. In this case, we interchange the colors of $wv$ and $wy$ as well as the colors of $xv$ and $xt$, and we color $vt$ with $7$.\medskip

\textbf{Case 2.} $b=9$.\medskip

$\quad$\textbf{Case 2.1.} $a=6$.

Note that  $9\in\{e,f\}$; otherwise we would color $vz$ with $9$ and $vt$ with $5$.
Moreover  $7\in\{e,f\}$; otherwise we would interchange  the colors of $wv$ and $wt$, and we color $vz$ with $7$ and $vt$ with $5$. Thus $\{e,f\}=\{7,9\}$. 
If $f=9$, then we interchange  the colors of $uv$ and $uz$, and color $vt$ with $4$. Therefore $f=7$. In this case, we interchange the colors of $uv$ and $uz$ as well as the colors of $wv$ and $wt$, and we color $vt$ with $4$.\medskip

$\quad$\textbf{Case 2.2.} $a=7$.

We may suppose that $9\in\{c,d\}$; otherwise we would color $vy$ with $9$ and $vt$ with $6$.
Similarly we obtain $9\in\{e,f\}$. It follows that $d=9$ as $b=9$, and so $f=9$.
If $e\neq 4$, then we would interchange   the colors of $uv$ and $uz$, and we color $vt$ with $4$. Thus we may assume that $e=4$. Let us now interchange  the colors of $xv$ and $xt$ as well as the colors of $wv$ and $wt$. We color $vz$ with $1$, and $vt$ with $5$.\medskip

$\quad$\textbf{Case 2.3.} $a\notin \{6,7\}$.

Then  $9\in\{c,d\}$; otherwise we would color $vy$ with $9$ and $vt$ with $6$.
It follows that $d=9$ as $b=9$. In this case we interchange  the colors of $wv$ and $wt$, and we color $vy$ with $7$ and $vt$ with $6$. 
\end{proof}

\begin{figure}[htb]
\centering  
\subfigure[]{
\label{fig:conf8-8383--38} 
\begin{tikzpicture}[scale=1]
\node [nod2] at (0,0) (v) [label={[xshift=.35cm, yshift=-0.65cm]{\scriptsize $8$}}] {};
\node at (0,0) (v) [label=above:{\scriptsize $v$}] {};
\node [nod2w] at (1.5,.65) (v1) [label=right:{\scriptsize $u$}] {}
	edge [] node[midway,below,font=\scriptsize] {$4$}(v);

\node [nod2] at (1.5,-.65) (v2) [label=right:{\scriptsize $z$}] {}
	edge [] node[midway,below,font=\scriptsize] {$5$} (v)
	edge [] node[midway,right,font=\scriptsize] {$f$} (v1);	
\node [nod2] at (.65,-1.5) (v3) [label=below:{\scriptsize $y$}] {}
	edge []  node[midway,left,font=\scriptsize] {$6$} (v);		
\node [nod2w] at (-.65,-1.5) (v4) [label=below:{\scriptsize $w$}]  {}
	edge [] node[midway,left,font=\scriptsize] {$7$} (v)
	edge [] node[midway,below,font=\scriptsize] {$c$} (v3);
\node [nod2] at (-1.5,-.65) (v5) [label=below:{\scriptsize $t$}]  {}
	edge [green!70!black] (v)
	edge []  node[midway,below,font=\scriptsize] {$b$}(v4);
\node [nod2w] at (-1.5,.65) (v6) [label=left:{\scriptsize $x$}] {}
	edge [] node[midway,above,font=\scriptsize] {$1$} (v)
	edge [] node[midway,left,font=\scriptsize] {$a$} (v5);	

\node [nod2w] at (-.65,1.5) (v7)   {}
	edge [] node[midway,right,font=\scriptsize] {$2$} (v);	
\node [nod2w] at (.65,1.5) (v8)  {}
	edge [] node[midway,right,font=\scriptsize] {$3$} (v);

\node [nod2w] at (2,-2) (v9) [label=right:{\scriptsize $p$}]  {}
	edge [] node[midway,below,font=\scriptsize] {$d$} (v3)	
	edge [] node[midway,right,font=\scriptsize] {$e$} (v2);						
\end{tikzpicture}}
\subfigure[]{
\label{fig:conf8-8383--3--38} 
\begin{tikzpicture}[scale=1]
\node [nod2] at (0,0) (v) [label={[xshift=.35cm, yshift=-0.67cm]{\scriptsize $8$}}] {};
\node at (0,0) (v) [label=above:{\scriptsize $v$}] {};
\node [nod2] at (1.5,.65) (v1) [label=right:{\scriptsize $u$}] {}
	edge [] node[midway,below,font=\scriptsize] {$4$}(v)
	edge [] node[midway,right,font=\scriptsize] {$h$} (v8);

\node [nod2] at (1.5,-.65) (v2) [label=right:{\scriptsize $z$}] {}
	edge [] node[midway,below,font=\scriptsize] {$5$} (v);	
\node [nod2] at (.65,-1.5) (v3) [label=below:{\scriptsize $y$}] {}
	edge []  node[midway,left,font=\scriptsize] {$6$} (v);		
\node [nod2w] at (-.65,-1.5) (v4) [label=below:{\scriptsize $w$}]  {}
	edge [] node[midway,left,font=\scriptsize] {$7$} (v)
	edge [] node[midway,below,font=\scriptsize] {$c$} (v3);
\node [nod2] at (-1.5,-.65) (v5) [label=left:{\scriptsize $t$}]  {}
	edge [green!70!black] (v)
	edge []  node[midway,below,font=\scriptsize] {$b$}(v4);
\node [nod2w] at (-1.5,.65) (v6) [label=left:{\scriptsize $x$}] {}
	edge [] node[midway,above,font=\scriptsize] {$1$} (v)
	edge [] node[midway,left,font=\scriptsize] {$a$} (v5);	

\node [nod2w] at (-.65,1.5) (v7)   {}
	edge [] node[midway,right,font=\scriptsize] {$2$} (v);	
\node [nod2w] at (.65,1.5) (v8) [label=above:{\scriptsize $r$}] {}
	edge [] node[midway,right,font=\scriptsize] {$3$} (v);

\node [nod2w] at (2,-2) (v9) [label=right:{\scriptsize $p$}]  {}
	edge [] node[midway,below,font=\scriptsize] {$d$} (v3)	
	edge [] node[midway,right,font=\scriptsize] {$e$} (v2);	
\node [nod2w] at (2.6,0) (v10) [label=right:{\scriptsize $q$}]  {}
	edge [] node[midway,above,font=\scriptsize] {$g$} (v1)	
	edge [] node[midway,below,font=\scriptsize] {$f$} (v2);							
\end{tikzpicture}}
\caption{ }
\end{figure}

\begin{lemma}\label{lem:8-8383--3--38}
The configuration in Figure \ref{fig:8-8383--3--38}  is reducible.
\end{lemma}
\begin{proof}

Suppose that $G$ contains this configuration. Denote by $v$ the $8$-vertex and by $t,y,z,u$ its $3$-neighbours.
Consider a proper $9$-coloring of $G-vt$, and uncolor the vertices $t,y,z,u$. We will find an appropriate color only for $vt$ by Remark \ref{rem:4-vert-later}.
We may suppose that the coloring is the one shown in Figure \ref{fig:conf8-8383--3--38}. Note that $9\in\{a,b\}$; otherwise $vt$ would have an available color. \medskip

\textbf{Case 1.} $a=9$.\medskip

$\quad$\textbf{Case 1.1.} $b=1$.

Observe that  $9\in\{c,d\}$; otherwise we would color $vy$ with $9$ and $vt$ with $6$.
Similarly we obtain $9\in\{e,f\}$ and $9\in\{g,h\}$. 
We may suppose that $d=7$; otherwise we would interchange   the colors of $xv$ and $xt$ as well as the colors of $wv$ and $wt$, and later we color $vy$ with $7$ and $vt$ with $6$. Moreover, $f=7$; otherwise we would interchange   the colors of $xv$ and $xt$ as well as the colors of $wv$ and $wt$, and later we color $vz$ with $7$ and $vt$ with $5$, where we note that $e\neq 7$ as $d=7$. Furthermore,  $h=7$; otherwise we would interchange   the colors of $xv$ and $xt$ as well as the colors of $wv$ and $wt$, and later we color $vu$ with $7$ and $vt$ with $4$, where we note that $g\neq 7$ as $f=7$. It follows that $g=9$ as $9\in\{g,h\}$. Let us now interchange  the colors of $xv$ and $xt$, the colors of $wv$ and $wt$, and the colors of $rv$ and $ru$. Then we can color $vt$ with $3$.\medskip

$\quad$\textbf{Case 1.2.} $b=6$.

Then  $9\in\{g,h\}$; otherwise we would color $vu$ with $9$ and $vt$ with $4$.
Moreover  $1\in\{g,h\}$; otherwise we would interchange  the colors of $xv$ and $xt$, and we color $vu$ with $1$ and $vt$ with $4$. Thus $\{g,h\}=\{1,9\}$. 
If $h=9$, then we interchange  the colors of $rv$ and $ru$, and color $vt$ with $3$. Therefore we may suppose that $h=1$. In this case, we interchange the colors of $rv$ and $ru$ as well as the colors of $xv$ and $xt$, and we color $vt$ with $3$.\medskip

$\quad$\textbf{Case 1.3.} $b\notin \{1,6\}$.

We may suppose that  $9\in\{c,d\}$; otherwise we would color $vy$ with $9$ and $vt$ with $6$.
Moreover  $1\in\{c,d\}$; otherwise we would interchange  the colors of $xv$ and $xt$, and we color $vy$ with $1$ and $vt$ with $6$. Thus $\{c,d\}=\{1,9\}$. 
If $c=9$, then we interchange  the colors of $wv$ and $wy$, and color $vt$ with $7$. Therefore we may suppose that $c=1$. In this case, we interchange the colors of $wv$ and $wy$ as well as the colors of $xv$ and $xt$, and we color $vt$ with $7$.\medskip

\textbf{Case 2.} $b=9$.\medskip

$\quad$\textbf{Case 2.1.} $a=6$.

Then  $9\in\{g,h\}$; otherwise we would color $vu$ with $9$ and $vt$ with $4$.
Moreover  $7\in\{g,h\}$; otherwise we would interchange  the colors of $wv$ and $wt$, and we color $vu$ with $7$ and $vt$ with $4$. Thus $\{g,h\}=\{7,9\}$. 
If $h=9$, then we interchange  the colors of $rv$ and $ru$, and color $vt$ with $3$. Therefore we may assume that $h=7$. In this case, we interchange the colors of $rv$ and $ru$ as well as the colors of $wv$ and $wt$, and we color $vt$ with $3$.\medskip

$\quad$\textbf{Case 2.2.} $a\neq 6$.

Note that  $9\in\{c,d\}$; otherwise we would color $vy$ with $9$ and $vt$ with $6$. It follows that $d=9$ as $b=9$. If $a\neq 7$, then we interchange  the colors of $wv$ and $wt$, and we color $vy$ with $7$ and $vt$ with $6$. Thus we may assume that $a=7$.
Similarly as above we obtain $9\in\{e,f\}$ and $9\in\{g,h\}$. It then follows that $d=f=h=9$ as $b=9$.
If $g\neq 3$, then we would interchange   the colors of $rv$ and $ru$, and we color $vt$ with $3$. Hence $g=3$. Let us now interchange  the colors of $xv$ and $xt$ as well as the colors of $wv$ and $wt$. We color $vu$ with $1$, and $vt$ with $4$.
\end{proof}

\begin{lemma}\label{lem:8-383--383}
The configuration in Figure \ref{fig:8-383--383}  is reducible.
\end{lemma}
\begin{proof}

Suppose that $G$ contains this configuration. Denote by $v$ the $8$-vertex and by $x,y,z,t$ its $3$-neighbours.
Consider a proper $9$-coloring of $G-vz$, and uncolor the vertices $x,y,z,t$. We will find an appropriate color only for $vz$ by Remark \ref{rem:4-vert-later}.
We may suppose that the coloring is the one shown in Figure \ref{fig:conf8-383--383}. Note that $9\in\{c,d\}$; otherwise $vz$ would have an available color. \medskip

\textbf{Case 1.} $c=9$.\medskip

$\quad$\textbf{Case 1.1.} $d=2$.

We may suppose that $9\in\{e,f\}$; otherwise we would color $vy$ with $9$ and $vz$ with $7$.  Moreover $1\in\{e,f\}$;  otherwise we would interchange the colors of $wv$ and $wz$, and later we color $vy$ with $1$ and $vz$ with $7$. Therefore we have $\{e,f\}=\{1,9\}$.
If $f=1$, then we interchange the colors of $uv$ and $uy$ as well as the colors of $wv$ and $wz$, and later we color $vz$ with $6$. If $f=9$, then we interchange the colors of $uv$ and $uy$, and we color $vz$ with $6$.\medskip

$\quad$\textbf{Case 1.2.} $d\neq 2$.

Then  $9\in\{a,b\}$; otherwise we would color $vt$ with $9$ and $vz$ with $2$. It follows that $a=9$ as  $c=9$.  If $d\neq 1$ then we would interchange the colors of $wv$ and $wz$, and we color $tv$ with $1$ and $vz$ with $2$. Thus we assume that $d=1$.
Similarly as above we obtain $9\in\{e,f\}$. 

Suppose first that $e=9$. If $f\neq 1$ then we interchange   the colors of $py$ and $pz$ as well as the colors of $wv$ and $wz$, and later we color $tv$ with $1$ and $vz$ with $2$. Thus we may assume that $f=1$. We interchange the colors of the edges $uv$ and $uy$. Let us alternate the color of the edges of the path $ypzwv$.   Now the color $6$ is available for $vz$, and so we color it.
Next we suppose that $f=9$. If $e\neq 6$ then we  would interchange the colors of $uv$ and $uy$, and we color $vz$ with $6$.  Thus we may assume that $e=6$. Now we interchange the colors of $wv$ and $wz$ as well as  the colors of $py$ and $pz$, later we color $vt$ with $1$ and $vz$ with $2$.\medskip

\textbf{Case 2.} $d=9$.\medskip

$\quad$\textbf{Case 2.1.} $c=5$.

Then  $9\in\{e,f\}$; otherwise we would color $vy$ with $9$ and $vz$ with $7$.
Similarly, $9\in\{a,b\}$. Since $d=9$, we deduce that $f=9$. 
If $e\neq 6$, then we interchange the colors of $uv$ and $uy$, and we color $vz$ with $6$. Thus we may assume that $e=6$. Note that $6\in\{a,b\}$; otherwise we would interchange the colors of $uv$ and $uy$ as well as  the colors of $py$ and $pz$, and later we color $vt$ with $6$ and $vz$ with $2$. Hence $\{a,b\}=\{6,9\}$.
If $b=6$ then we  interchange the colors of $wv$ and $wt$, the colors of $uv$ and $uy$,  the colors of $py$ and $pz$,  and later we color $vz$ with $1$. Thus we obtain that $b=9$. In this case we interchange the colors of $wv$ and $wt$, and we color $vz$ with $1$. \medskip

$\quad$\textbf{Case 2.2.} $c=6$.

We may suppose that  $9\in\{e,f\}$; otherwise we would color $vy$ with $9$ and $vz$ with $7$.
Similarly, $9\in\{g,h\}$. It then follows that $f=h=9$ as $d=9$. 
If $e\neq 6$, then we interchange the colors of $uv$ and $uy$, and observe that this  reduces to Case 1 as the color  of $wz$ is missing color around $v$. Thus we may assume that $e=6$. Let us now interchange the colors of $uv$ and $uy$, the colors of $py$ and $pz$, and the colors of $wv$ and $wz$. We again observe that this  reduces to Case 1 as the color of $wz$ is missing color around $v$. \medskip

$\quad$\textbf{Case 2.3.} $c=7$.

Observe that $9\in\{g,h\}$; otherwise we would color $vx$ with $9$ and $vz$ with $5$.
If $f\neq 9$ then we interchange the colors of $py$ and $pz$, and we color $vz$ with $9$, where we note that $e\neq 7$. Thus we obtain $f=9$. It follows that $h=9$ as $9\in\{g,h\}$. We may suppose that $e=6$; otherwise we would interchange the colors of $uv$ and $uy$, and we color $vz$ with $6$.
Let us now interchange the colors of $uv$ and $uy$ as well as  the colors of $py$ and $pz$. We color $vx$ with $6$ and $vz$ with $5$.\medskip

$\quad$\textbf{Case 2.4.} $c\notin \{5,6,7\}$.

We may suppose that $9\in\{e,f\}$; otherwise we would color $vy$ with $9$ and $vz$ with $7$.
Similarly, $9\in\{g,h\}$. It then follows that $f=h=9$ as $d=9$. 
If $e\neq 6$ then we interchange the colors of $uv$ and $uy$, and we color $vz$ with $6$. Thus we obtain $e=6$. Now we  interchange the colors of $uv$ and $uy$ as well as  the colors of $py$ and $pz$. Later, we color $vx$ with $6$, and $vz$ with $5$.
\end{proof}

\begin{figure}[htb]
\centering  
\subfigure[]{
\label{fig:conf8-383--383} 
\begin{tikzpicture}[scale=1]
\node [nod2] at (0,0) (v) [label={[xshift=.35cm, yshift=-0.65cm]{\scriptsize $8$}}] {};
\node at (0,0) (v) [label=above:{\scriptsize $v$}] {};
\node [nod2w] at (1.5,.65) (v1) [label=right:{\scriptsize $u$}] {}
	edge [] node[midway,below,font=\scriptsize] {$6$}(v);
	
\node [nod2] at (1.5,-.65) (v2) [label=right:{\scriptsize $y$}] {}
	edge [] node[midway,below,font=\scriptsize] {$7$} (v)
	edge [] node[midway,right,font=\scriptsize] {$f$} (v1);	
\node [nod2] at (.65,-1.5) (v3) [label=below:{\scriptsize $z$}] {}
	edge [green!70!black]  (v);		
\node [nod2w] at (-.65,-1.5) (v4) [label=below:{\scriptsize $w$}]  {}
	edge [] node[midway,left,font=\scriptsize] {$1$} (v)
	edge [] node[midway,below,font=\scriptsize] {$c$} (v3);
\node [nod2] at (-1.5,-.65) (v5) [label=below:{\scriptsize $t$}]  {}
	edge []  node[midway,above,font=\scriptsize] {$2$}(v)
	edge []  node[midway,below,font=\scriptsize] {$b$}(v4);
\node [nod2w] at (-1.5,.65) (v6)  {}
	edge [] node[midway,above,font=\scriptsize] {$3$} (v);	

\node [nod2w] at (-.65,1.5) (v7)   {}
	edge [] node[midway,right,font=\scriptsize] {$4$} (v);	
\node [nod2] at (.65,1.5) (v8) [label=left:{\scriptsize $x$}] {}
	edge [] node[midway,right,font=\scriptsize] {$5$} (v)
	edge [] node[midway,right,font=\scriptsize] {$g$} (v1);		
	
\node [nod2w] at (-2,-1.8) (v9)   {}
	edge [] node[midway,left,font=\scriptsize] {$a$} (v5);		

\node [nod2w] at (2,-2) (v9) [label=right:{\scriptsize $p$}]  {}
	edge [] node[midway,below,font=\scriptsize] {$d$} (v3)	
	edge [] node[midway,right,font=\scriptsize] {$e$} (v2);		
		
\node [nod2w] at (1,2.5) (v9)   {}
	edge [] node[midway,left,font=\scriptsize] {$h$} (v8);						
\end{tikzpicture}}
\subfigure[]{
\label{fig:conf8-83--38--2}
\begin{tikzpicture}[scale=1]
\node [nod2] at (0,0) (v) [label={[xshift=.35cm, yshift=-0.65cm]{\scriptsize $8$}}] {};
\node at (0,0) (v) [label=above:{\scriptsize $v$}] {};
\node [nod2w] at (1.5,.65) (v1) [label=right:{\scriptsize $r$}] {}
	edge [] node[midway,below,font=\scriptsize] {$6$}(v);

\node [nod2] at (1.5,-.65) (v2) [label=right:{\scriptsize $t$}] {}
	edge [] node[midway,below,font=\scriptsize] {$7$} (v)
	edge [] node[midway,right,font=\scriptsize] {$e$} (v1);	
\node [nod2] at (.65,-1.5) (v3) [label=below:{\scriptsize $z$}] {}
	edge [green!70!black]  (v);		
\node [nod2w] at (-.65,-1.5) (v4) [label=below:{\scriptsize $y$}]  {}
	edge [] node[midway,left,font=\scriptsize] {$1$} (v)
	edge [] node[midway,below,font=\scriptsize] {$b$} (v3);
\node [nod2w] at (-1.5,-.65) (v5)  {}
	edge []  node[midway,above,font=\scriptsize] {$2$}(v);
\node [nod2] at (-1.5,.65) (v6) [label=below:{\scriptsize $x$}]  {}
	edge [] node[midway,above,font=\scriptsize] {$3$} (v);	

\node [nod2w] at (-.65,1.5) (v7)   {}
	edge [] node[midway,right,font=\scriptsize] {$4$} (v);	
\node [nod2w] at (.65,1.5) (v8) {}
	edge [] node[midway,right,font=\scriptsize] {$5$} (v);		
	
\node [nod2w] at (-1.6,1.8) (v9)   {}
	edge [] node[midway,left,font=\scriptsize] {$a$} (v6);		

\node [nod2w] at (2,-2) (v9) [label=right:{\scriptsize $p$}]  {}
	edge [] node[midway,below,font=\scriptsize] {$c$} (v3)	
	edge [] node[midway,right,font=\scriptsize] {$d$} (v2);						
\end{tikzpicture}}
\subfigure[]{
\label{fig:conf8-83-2--38}
\begin{tikzpicture}[scale=1]
\node [nod2] at (0,0) (v) [label={[xshift=.35cm, yshift=-0.65cm]{\scriptsize $8$}}] {};
\node at (0,0) (v) [label=above:{\scriptsize $v$}] {};
\node [nod2] at (1,.45) (v1) [label=above:{\scriptsize $u$}] {}
	edge [green!70!black]  (v);

\node [nod2] at (1.5,-.65) (v2) [label=right:{\scriptsize $w$}] {}
	edge [] node[midway,below,font=\scriptsize] {$7$} (v);
\node [nod2w] at (.65,-1.5) (v3) [label=below:{\scriptsize $y$}]  {}
	edge  node[midway,left,font=\scriptsize] {$6$} (v)
	edge node[midway,below,font=\scriptsize] {$b$}  (v2);		
\node [nod2w] at (-.65,-1.5) (v4)   {}
	edge [] node[midway,left,font=\scriptsize] {$5$} (v);
\node [nod2w] at (-1.5,-.65) (v5)  [label=left:{\scriptsize $t$}] {}
	edge []  node[midway,above,font=\scriptsize] {$4$}(v);
\node [nod2] at (-1.5,.65) (v6) [label=left:{\scriptsize $z$}] {}
	edge [] node[midway,above,font=\scriptsize] {$3$} (v)
	edge []  node[midway,left,font=\scriptsize] {$c$} (v5);	

\node [nod2w] at (-.65,1.5) (v7)   {}
	edge [] node[midway,right,font=\scriptsize] {$2$} (v);	
\node [nod2w] at (.65,1.5) (v8)  {}
	edge [] node[midway,right,font=\scriptsize] {$1$} (v);

\node [nod2w] at (1.8,.8) (v9) [label=right:{\scriptsize $x$}]  {}
	edge [] node[midway,below,font=\scriptsize] {$9$}(v1)
	edge [] node[midway,right,font=\scriptsize] {$a$} (v2);		

\node [nod2w] at (-1.6,1.8) (v10)   {}
	edge [] node[midway,left,font=\scriptsize] {$d$} (v6);	

\end{tikzpicture}}
\caption{ }
\end{figure}

\begin{lemma}\label{lem:8-83--38--2}
The configuration in Figure \ref{fig:8-83--38--2}  is reducible.
\end{lemma}
\begin{proof}

Suppose that $G$ contains this configuration. Denote by $v$ the $8$-vertex, by $x$ its $2$-neighbour and by $z,t$ its $3$-neighbours.
Consider a proper $9$-coloring of $G-vz$, and uncolor the vertices $x,z,t$. We will find an appropriate color only for $vz$ by Remark \ref{rem:4-vert-later}.
We may suppose that the coloring is the one shown in Figure \ref{fig:conf8-83--38--2}. Note that $9\in\{b,c\}$; otherwise $vz$ would have an available color. \medskip

\textbf{Case 1.} $b=9$.\medskip

$\quad$\textbf{Case 1.1.} $c=3$.

Then $9\in\{d,e\}$; otherwise we would color $vt$ with $9$ and $vz$ with $7$. Moreover, we suppose that $1\in\{d,e\}$; otherwise we would  interchange   the colors of $yv$ and $yz$, and we color $vt$ with $1$ and $vz$ with $7$. Hence, $\{d,e\}=\{1,9\}$. 
If $e=9$ then we interchange   the colors of $rv$ and $rt$, and we color $vz$ with $6$. Thus $e=1$ and $d=9$. Let us interchange the colors of $yv$ and $yz$ as well as the colors of $pz$ and $pt$, and the colors of $rv$ and $rt$. Now the color $6$ is available for $vz$, and so we color it.\medskip

$\quad$\textbf{Case 1.2.} $c\neq 3$.

Observe that $a=9$; otherwise we would color $vx$ with $9$ and $vz$ with $3$. If $c\neq 1$ then we interchange the colors of $yv$ and $yz$, and we color $vx$ with $1$ and $vz$ with $3$. Hence we assume that $c=1$. On the other hand, if $9\notin\{d,e\}$ then we would color $vt$ with $9$ and $vz$ with $7$. Thus we may assume that $9\in\{d,e\}$.
First suppose that $d=9$. If $e\neq1$ then we interchange  the colors of $yv$ and $yz$ as well as the colors of $pz$ and $pt$, and we color $vx$ with $1$ and $vz$ with $3$. Hence, $e=1$. Let us interchange the colors of $rv$ and $rt$, the colors of $yv$ and $yz$, the colors of $pz$ and $pt$. Now we can color $vz$ with $6$.
Next we suppose that $e=9$. If $d\neq 6$ then we interchange  the colors of $rv$ and $rt$, and we color $vz$ with $6$. Hence, $d=6$. Let us interchange the colors of $pz$ and $pt$ as well as the colors of $yv$ and $yz$. Now we color $vx$ with $1$ and $vz$ with $3$.\medskip

\textbf{Case 2.} $c=9$.\medskip

$\quad$\textbf{Case 2.2.} $b=3$.

We may suppose that  $9\in\{d,e\}$; otherwise we would color $vt$ with $9$ and $vz$ with $7$. Then $e=9$ as $c=9$. If $d\neq 6$, then we interchange the colors of $rv$ and $rt$, and we color $vz$ with $6$. Thus $d=6$. 
Also, we may assume that $a=9$; otherwise we would color $vx$ with $9$, $vt$ with $3$ and $vz$ with $7$. Now we interchange the colors of $rv$ and $rt$ as well as the colors of $pz$ and $pt$. It follows that we can color $vx$ with $6$, $vt$ with $3$ and $vz$ with $7$.\medskip

$\quad$\textbf{Case 2.1.} $b=7$.

Observe that $a=9$; otherwise we would color $vx$ with $9$ and $vz$ with $3$.
Also, we  may suppose that $e=9$; otherwise we would interchange the colors of $pz$ and $pt$ (note that $d\neq 7$), and we color $vz$ with $9$. If $d\neq 6$, then we interchange the colors of $rv$ and $rt$, and we color $vz$ with $6$. Thus $d=6$.
Let us now interchange the colors of $rv$ and $rt$ as well as the colors of $pz$ and $pt$. It follows that we can color $vx$ with $6$ and $vz$ with $3$.\medskip

$\quad$\textbf{Case 2.3.} $b\notin \{3,7\}$.

Observe that $a=9$; otherwise we would color $vx$ with $9$ and $vz$ with $3$.
In addition $9\in\{d,e\}$; otherwise we would color $vt$ with $9$ and $vz$ with $7$. Since $c=9$, we obtain that  $e=9$. If $d\neq 6$, then we interchange the colors of $rv$ and $rt$, and we color $vx$ with $6$ and $vz$ with $3$. Thus $d=6$. Let us alternate the colors of the edges of the path $vrtpz$.
If $b\neq 6$, then we color $vx$ with $6$ and $vz$ with $3$. 
If $b=6$, then we interchange the colors of $yv$ and $yz$, and later we color $vx$ with $1$, $vz$ with $3$.
\end{proof}

\begin{lemma}\label{lem:8-83-2--38}
The configuration in Figure \ref{fig:8-83-2--38}  is reducible.
\end{lemma}
\begin{proof}
Suppose that $G$ contains this configuration. Denote by $v$ the $8$-vertex, by $u$ its $2$-neighbour and by $w,z$ its $3$-neighbours.
Consider a proper $9$-coloring of $G-vu$, and uncolor the vertices $w,u,z$. We will find an appropriate color only for $vu$ by Remark \ref{rem:4-vert-later}.

We may suppose that the coloring is the one shown in Figure \ref{fig:conf8-83-2--38}. Notice first that 
$b=9$; otherwise we would color $vw$ with $9$ and $vu$ with   $7$.
If $a\neq 6$, then we interchange the colors of $yv$ and $yw$, and we color $vu$ with $6$. Therefore, we may assume that $a=6$. Consider now the neighbours of $z$. If $9\notin \{c,d\}$ then we would color $vz$ with $9$, and $vu$ with $3$. Hence $9\in \{c,d\}$.
First suppose that $c=9$. If $d\neq 4$, then  we would interchange the colors of $tv$ and $tz$, and we color $vu$ with $4$. Thus $d=4$, and so we interchange the colors of $yv$ and $yw$ as well as the colors of $xu$ and $xw$. Now, we color $vz$ with $6$, and $vu$ with $3$. 
We next suppose that $d=9$. Let us interchange the colors of $yv$ and $yw$ as well as the colors of $xu$ and $xw$.  If $c\neq 6$, then  we color $vz$ with $6$, and $vu$ with $3$. If $c=6$, then  we interchange the colors of $tv$ and $tz$, and we color $vu$ with $4$. 
\end{proof}

In the rest of the paper, we will apply the discharging method to show that $G$ does not exist. We assign to each vertex $v$ a charge $\mu(v)=d(v)-4$ and to each face $f$ a charge $\mu(f)=\ell(f)-4$. By Euler's formula, the total charge is 
$$\sum_{v\in V}\mu(v)+\sum_{f\in F}\mu(f)=\sum_{v\in V}\left(d(v)-4\right)+\sum_{f\in F}(\ell(f)-4)=-8.$$

We next present some rules and redistribute the charges accordingly. Once the discharging finishes, we will show that the final charge $\mu^*(v)\geq 0$ and $\mu^*(f)\geq 0$ for each $v\in V(G)$ and $f\in F(G)$, contradicting the fact that the total charge is $-8$.

\subsection{Discharging Rules} \label{sub:} ~\medskip

We apply the following discharging rules.

\begin{itemize}
\setlength\itemsep{.5em}
\item[\textbf{R1:}] Every $2$-vertex receives $1$  from each of its neighbours.

\item[\textbf{R2:}] Every $3$-vertex receives $\frac{1}{3}$  from each of its neighbours.

\item[\textbf{R3:}] Every $5$-vertex sends $\frac{1}{3}$ to each incident $3$-face.

\item[\textbf{R4:}] Every $6^+$-vertex sends $\frac{1}{2}$ to each incident $3$-face containing a $4^-$-vertex and $\frac{1}{3}$ to other $3$-faces.

\item[\textbf{R5:}] Every $5^+$-face transfers its positive charge equally to its incident $8$-vertices.\medskip
\end{itemize}

\noindent
\textbf{Checking} $\mu^*(v), \mu^*(f)\geq 0$, for $v\in V(G), f\in F(G)$\medskip

\noindent
We initially show that $\mu^*(f)\geq 0$ for each $f\in F(G)$. Let $f\in F(G)$ be a face. If $f$ is a $4^+$-face, then $\mu(f)=\mu^*(f)=\ell(f)-4\geq 0$ by R5. Now we suppose that $f$ is a $3$-face $uvw$ with $d(u)\leq d(v)\leq d(w)$. The initial charge of $f$ is $\mu^*(f)=\ell(f)-4=-1$. 
If $d(u)\leq 4$, then, by Lemma \ref{lem:uv-10}, $v$ and $w$ are $6^+$-vertices. It follows that $f$ receives  $\frac{1}{2}$ from each of $v,w$ by R4, and so  $\mu^*(f)=-1+ 2\times \frac{1}{2}= 0$. 
If  $d(u)\geq 5$, then $f$ receives $\frac{1}{3}$ from each of $u,v,w$ by R3-R4,  and so  $\mu^*(f)=-1+3\times \frac{1}{3}= 0$.\bigskip

\noindent
Now let us show that  $\mu^*(v)\geq 0$ for each $v\in V(G)$.  We pick a vertex $v\in V(G)$ with $d(v)=k$. By Lemma \ref{lem:min-deg}, we have $k\geq 2$. \medskip

\textbf{(1).} Let $k\leq 3$.  By Lemma \ref{lem:uv-10}, each neighbour of $v$ is a $7^+$-vertex. If $v$ is a $2$-vertex, then $v$ receives $1$ from each of its neighbours by R1, and so $\mu^*(v)=d(v)-4+2\times 1= 0$. If $v$ is a $3$-vertex, then $v$ receives $\frac{1}{3}$ from each of its neighbours by R2, and so $\mu^*(v)=d(v)-4+3\times \frac{1}{3}= 0$.\medskip

\textbf{(2).} Let $k=4$. Note that $v$ neither receives nor sends any charge, and so  $\mu^*(v)=\mu(v)=d(v)-4= 0$.\medskip

\textbf{(3).} Let $k=5$. The initial charge of $v$ is $\mu(v)=1$. Since $G$ has no $4$-fan, we deduce that $m_3(v)\leq 3$. Then $v$ sends  $\frac{1}{3}$ to each incident $3$-face, and so $\mu^*(v)\geq 1-3\times \frac{1}{3}= 0$.\medskip 

\textbf{(4).} Let $k=6$. The initial charge of $v$ is $\mu(v)=2$. Each neighbour of $v$ is a $4^+$-vertex by Lemma \ref{lem:uv-10}. Since $G$ has no $4$-fan, we deduce that $m_3(v)\leq 4$. Then  $v$ sends at most $\frac{1}{2}$ to each incident $3$-face by R4, and so $\mu^*(v)\geq 2-4\times \frac{1}{2}= 0$. \medskip

\textbf{(5).} Let $k=7$.  The initial charge of $v$ is $\mu(v)=3$. Each neighbour of $v$ is a $3^+$-vertex by Lemma \ref{lem:uv-10}.  Since $G$ has no $4$-fan, we deduce that $m_3(v)\leq 5$.
If $m_3(v)\leq 1$, then   $\mu^*(v)\geq 3-\frac{1}{2}- 7\times \frac{1}{3}> 0$ after $v$ sends  at most $\frac{1}{2}$ to each incident $3$-face by R4,   $\frac{1}{3}$ to each $3$-neighbour by R2.  
If $m_3(v)=2$, then $v$ has at most six $3$-neighbours by  Lemma \ref{lem:uv-10}, and so $\mu^*(v)\geq 3-2\times\frac{1}{2}- 6\times \frac{1}{3}= 0$ after $v$ sends  at most $\frac{1}{2}$ to each incident $3$-face by R4,   $\frac{1}{3}$ to each $3$-neighbour by R2. Therefore we may suppose that $m_3(v)\geq 3$.  
If $m_3(v)=3$, then, by  Lemmas \ref{lem:uv-10} and \ref{lem:7-has-no-two-3-on-3face}, $v$ has at most four $3$-neighbours, and so  $\mu^*(v)\geq 3-3\times\frac{1}{2}- 4\times \frac{1}{3}> 0$ after $v$ sends  at most $\frac{1}{2}$ to each incident $3$-face by R4,   $\frac{1}{3}$ to each $3$-neighbour by R2.
If $m_3(v)=4$, then, similarly as above, by  Lemmas \ref{lem:uv-10} and \ref{lem:7-has-no-two-3-on-3face}, $v$ has at most two $3$-neighbours since $G$ has no $4$-fan, and so $\mu^*(v)\geq 3-4\times\frac{1}{2}- 2\times \frac{1}{3}>0$ after  $v$ sends  at most $\frac{1}{2}$ to each incident $3$-face by R4,   $\frac{1}{3}$ to each $3$-neighbour by R2. 
Finally, if $m_3(v)=5$, then by  Lemmas \ref{lem:uv-10} and \ref{lem:7-has-no-two-3-on-3face}, the fact that $G$ has no $4$-fan implies that $v$ has at most one $3$-neighbour, and so $\mu^*(v)\geq 3-5\times\frac{1}{2}-  \frac{1}{3}>0$ after $v$ sends  at most $\frac{1}{2}$ to each incident $3$-face by R4,  $\frac{1}{3}$ to each $3$-neighbour by R2.  \medskip 

%
%
%
%
%

\textbf{(6).} Let $k=8$. The initial charge of $v$ is $\mu(v)=4$, and $v$ has at most one $2$-neighbour by Lemma \ref{lem:8-has-one-2}. Since $G$ has no $4$-fan, we deduce that $m_3(v)\leq 6$.\medskip

We first suppose that $v$ has no $2$-neighbour. If $m_3(v)\leq 2$, then $\mu^*(v)\geq 4-2\times \frac{1}{2}-8\times \frac{1}{3}> 0$  after $v$ sends at most $\frac{1}{2}$ to each incident $3$-face by R4,  $\frac{1}{3}$ to each $3$-neighbour by R2.
If $3\leq m_3(v)\leq 4$ then $v$ has at most six $3$-neighbours by Lemma \ref{lem:uv-10}, and so $\mu^*(v)\geq 4-4\times \frac{1}{2}-6\times \frac{1}{3}= 0$  after $v$ sends at most $\frac{1}{2}$ to each incident $3$-face by R4,  $\frac{1}{3}$ to each $3$-neighbour by R2. 
Thus we may assume that $ m_3(v)\geq 5$. First suppose that $ m_3(v)=5$. In such a case,  $v$ has at most five $3$-neighbours by Lemma \ref{lem:uv-10}. If $v$  has at most four $3$-neighbours, then $\mu^*(v)\geq 4-5\times \frac{1}{2}-4\times \frac{1}{3}> 0$  after $v$ sends at most  $\frac{1}{2}$ to each incident $3$-face by R4,  $\frac{1}{3}$ to each $3$-neighbour by R2. 
Thus we may assume that $v$ has exactly five $3$-neighbours. If $v$ is incident to a $(5^+,5^+,8)$-triangle $f$, then $\mu^*(v)\geq 4-4\times \frac{1}{2}-\frac{1}{3}-5\times \frac{1}{3}= 0$  after $v$ sends at most  $\frac{1}{2}$ to each incident $3$-face other than $f$ by R4, $\frac{1}{3}$ to $f$ by R4,  $\frac{1}{3}$ to each $3$-neighbour by R2. 
Thus we may assume that $v$ is not incident to any $(5^+,5^+,8)$-triangle. In this case, we deduce that $v$ is incident to a $5^+$-face $f$ containing two $3$-neighbours of $v$ since the configuration in Figures \ref{fig:8-8383--3--38} and  \ref{fig:8-383--383} are reducible. By R5, $f$ sends at least $\frac{1}{3}$ to $v$. Thus $\mu^*(v)\geq 4+\frac{1}{3}-5\times \frac{1}{2}-5\times \frac{1}{3}> 0$  after $v$ sends at most $\frac{1}{2}$ to each incident $3$-face by R4,  $\frac{1}{3}$ to each $3$-neighbour by R2.

Finally suppose that $ m_3(v)=6$. Since $G$ has no $4$-fan, we deduce that  $v$ is incident a pair of three consecutive $3$-faces, and so $v$ has at most four $3$-neighbours by Lemma \ref{lem:uv-10}.  If $v$  has at most three $3$-neighbours, then $\mu^*(v)\geq 4-6\times \frac{1}{2}-3\times \frac{1}{3}= 0$  after $v$ sends at most  $\frac{1}{2}$ to each incident $3$-face by R4,  $\frac{1}{3}$ to each $3$-neighbour by R2. Thus we may assume that $v$ has exactly four $3$-neighbours. 
We note that if $v$ is incident to two $(5^+,5^+,8)$-triangles $f_1,f_2$, then $\mu^*(v)\geq 4-4\times \frac{1}{2}-2\times \frac{1}{3}-4\times \frac{1}{3}= 0$  after $v$ sends at most  $\frac{1}{2}$ to each incident $3$-faces other than $f_1,f_2$ by R4, $\frac{1}{3}$ to each of $f_1,f_2$ by R4,  $\frac{1}{3}$ to each $3$-neighbour by R2. Thus we may assume that $v$ is incident to at most one $(5^+,5^+,8)$-triangle. It then follows from Lemma \ref{lem:uv-10} that $v$ is incident to a $3$-vertex on two $3$-faces. Note that there are no two $3$-neighbours $u,w$ of $v$ such that each of them is adjacent to $v$ on two $3$-faces by Lemma \ref{lem:8-v-has-no-two-3-on-3-faces}. This implies that there exist exactly three $3$-neighbours of $v$, say $v_1,v_2,v_3$,  such that each $vv_i$, for $1\leq i \leq 3$, is contained in a $4^+$-face.
Since the configuration in Figure \ref{fig:8-8383--38} is reducible, we deduce that  $v$ is incident to a $5^+$-face $f$ containing two $3$-neighbours of $v$. By R5, $f$ sends at least $\frac{1}{3}$ to $v$.  Thus $\mu^*(v)\geq 4+\frac{1}{3}-6\times \frac{1}{2}-4\times \frac{1}{3}= 0$  after $v$ sends  at most $\frac{1}{2}$ to each incident $3$-face by R4,  $\frac{1}{3}$ to each $3$-neighbour by R2.  \medskip

Let us now suppose that $v$ has a $2$-neighbour. Denote by $v_1$ the $2$-neighbour of $v$. \medskip

$\quad$\textbf{(6.1).} Let $m_3(v)\leq 2$. If $m_3(v)=0$ then $\mu^*(v)\geq 4-1-7\times \frac{1}{3}> 0$  after $v$ sends  $1$ to its $2$-neighbour by R1, $\frac{1}{3}$ to each $3$-neighbour by R2. If $m_3(v)\geq 1$, then, by Lemma \ref{lem:uv-10}, $v$ has at most six $3$-neighbours, and so $\mu^*(v)\geq 4-2\times\frac{1}{2}-1-6\times \frac{1}{3}= 0$  after $v$ sends  at most $\frac{1}{2}$ to each incident $3$-face by R4, $1$ to its $2$-neighbour by R1, $\frac{1}{3}$ to each $3$-neighbour by R2.\medskip

$\quad$\textbf{(6.2).} Let $m_3(v)=3$. By Lemma \ref{lem:uv-10}, $v$ has at most five $3$-neighbours.  If $v$  has at most four $3$-neighbours, then  $\mu^*(v)\geq 4-3\times\frac{1}{2}-1-4\times \frac{1}{3}> 0$  after $v$ sends at most  $\frac{1}{2}$ to each incident $3$-face by R4, $1$ to its $2$-neighbour by R1, $\frac{1}{3}$ to each $3$-neighbour by R2. 
Thus we may assume that $v$ has exactly five $3$-neighbours. 
If $vv_1$ is contained in a $3$-face $f$, then $v$ is not adjacent to a $3$-vertex on a $3$-face   by Lemma \ref{lem:8-no-2-on3face-and-3-on3face}, and so we infer that $v$ can have at most four $3$-neighbours, contradicting the fact that $v$ has five $3$-neighbours. Hence, $vv_1$ is not contained in a $3$-face.

If  $v$ is incident to a $(5^+,5^+,8)$-triangle $f$, then $\mu^*(v)\geq 4-2\times \frac{1}{2}-\frac{1}{3}-1-5\times \frac{1}{3}= 0$  after $v$ sends at most $\frac{1}{2}$ to each incident $3$-face other than $f$ by R4, $\frac{1}{3}$ to  $f$ by R4, $1$ to its $2$-neighbour by R1, $\frac{1}{3}$ to each $3$-neighbour by R2. Thus we may assume that each $3$-face incident to $v$ contains a $4^-$-vertex.  We note that, by Lemma \ref{lem:uv-10},  $v$ is incident to two adjacent $3$-faces since $G$ has exactly five $3$-neighbours. 
In this case,  $v$ is incident to a $5^+$-face $f$ containing two $3^-$-neighbours of $v$ by Lemmas \ref{lem:uv-10} and \ref{lem:8-has-no-233383}. By R5, $f$ sends at least $\frac{1}{3}$ to $v$. Thus  $\mu^*(v)\geq 4+\frac{1}{3}-3\times\frac{1}{2}-1-5\times \frac{1}{3}> 0$  after $v$ sends  at most $\frac{1}{2}$ to each incident $3$-face by R4, $1$ to its $2$-neighbour by R1, $\frac{1}{3}$ to each $3$-neighbour by R2. \medskip

$\quad$\textbf{(6.3).} Let $m_3(v)=4$. By Lemma \ref{lem:uv-10}, $v$ has at most five $3$-neighbours. If $v$  has at most three $3$-neighbours, then  $\mu^*(v)\geq 4-4\times\frac{1}{2}-1-3\times \frac{1}{3}= 0$  after $v$ sends at most $\frac{1}{2}$ to each incident $3$-face by R4, $1$ to its $2$-neighbour by R1, $\frac{1}{3}$ to each $3$-neighbour by R2. Thus we may assume that $v$ has either four or five $3$-neighbours.
If $vv_1$ is contained in a $3$-face $f$, then $v$ is not adjacent to a $3$-vertex on a $3$-face   by Lemma \ref{lem:8-no-2-on3face-and-3-on3face}, and so we infer that $v$ can have at most two $3$-neighbours as $G$ has no $4$-fan, contradicting the fact that $v$ has at least four $3$-neighbours. Hence, $vv_1$ is not contained in a $3$-face. On the other hand, we note that  $v$ is not adjacent to any $3$-vertex on two $3$-faces by Lemma \ref{lem:8-v-has-no2-when-3-diamond}, i.e., each $3$-neighbour of $v$ is incident a $4^+$-face containing $v$.

First suppose that $n_3(v)=4$.  Since $G$ has no $4$-fan, there are three possible cases by Lemma \ref{lem:uv-10}: $v$ is incident to either two non-adjacent $3$-faces and two adjacent $3$-faces or one $3$-face and three consecutive $3$-faces or a pair of two adjacent $3$-faces. We consider each case to show that $v$ is incident to a $5^+$-face containing two $3^-$-neighbours of $v$.
First, if $v$ has two non-adjacent $3$-faces and two adjacent $3$-faces then  $v$ must be incident to a $5^+$-face $f$ containing two $3^-$-neighbours of $v$, since the configurations in Figures \ref{fig:8-83--38--2} and \ref{fig:8-83-2--38} are reducible. 
Next, if $v$ is incident to one $3$-face and three consecutive  $3$-faces then $v$ is incident to a $5^+$-face $f$ containing two $3^-$-neighbours of $v$, since the configurations in Figure \ref{fig:8-83--38--2} and \ref{fig:8-83-2--38} are reducible. 
Finally, if $v$ is incident to a pair of two adjacent $3$-faces then $v$ is incident to a $5^+$-face $f$ containing two $3^-$-neighbours of $v$ since the configurations in Figure \ref{fig:8-83--38--2} and \ref{fig:8-83-2--38} are reducible. Consequently, by R5, the face $f$ sends at least $\frac{1}{3}$ to $v$. Therefore  $\mu^*(v)\geq 4+\frac{1}{3}-4\times\frac{1}{2}-1-4\times \frac{1}{3}= 0$ after $v$ sends at most  $\frac{1}{2}$ to each incident $3$-face by R4, $1$ to its $2$-neighbour by R1, $\frac{1}{3}$ to each $3$-neighbour by R2. 

Suppose next that $n_3(v)=5$. By Lemma \ref{lem:uv-10}, we observe that  $v$ is incident to a pair of two adjacent $3$-faces. It then follows that  $v$ is incident to at least two $5^+$-faces containing two $3^-$-neighbours of $v$  since the configurations in Figure \ref{fig:8-83--38--2} and \ref{fig:8-83-2--38} are reducible. By R5, each of those faces sends at least $\frac{1}{3}$ to $v$. Thus  $\mu^*(v)\geq 4+2\times\frac{1}{3}-4\times\frac{1}{2}-1-5\times \frac{1}{3}= 0$ after $v$ sends at most $\frac{1}{2}$ to each incident $3$-face by R4, $1$ to its $2$-neighbour by R1, $\frac{1}{3}$ to each $3$-neighbour by R2.

$\quad$\textbf{(6.4).} Let $m_3(v)=5$. By Lemma \ref{lem:uv-10}, $v$ has at most four $3$-neighbours. If $v$  has at most one $3$-neighbour, then  $\mu^*(v)\geq 4-5\times\frac{1}{2}-1- \frac{1}{3}> 0$  after $v$ sends at most  $\frac{1}{2}$ to each incident $3$-face by R4, $1$ to its $2$-neighbour by R1, $\frac{1}{3}$ to each $3$-neighbour by R2. Thus we may assume that $v$ has at least two $3$-neighbours.
If $vv_1$ is contained in a $3$-face $f$, then it follows from Lemma \ref{lem:8-no-2-on3face-and-3-on3face} that $v$ is not adjacent to any $3$-vertex on a $3$-face  since $G$ has no $4$-fan. Therefore we infer that $v$ can have at most one $3$-neighbour, a contradiction to  the fact that $v$ has at least two $3$-neighbours. Hence, $vv_1$ is not contained in a $3$-face. Note that,  by Lemma \ref{lem:8-v-has-no2-when-3-diamond}, $v$ is not adjacent to any $3$-vertex on two $3$-faces. This implies that each $3$-neighbour of $v$ is incident a $4^+$-face containing $v$, and so $v$ has at most four $3$-neighbours.  Since $G$ has no $4$-fan, and $v_1$ does not belong to any $3$-face, we deduce that $v$ is incident to both two consecutive $3$-faces and three consecutive $3$-faces.

Let $n_3(v)=2$. Then $v$ is incident to at least one $5^+$-face containing two $3^-$-neighbours of $v$ since the configuration in Figures \ref{fig:8-83--38--2} and \ref{fig:8-83-2--38} are reducible. By R5, this face sends at least $\frac{1}{3}$ to $v$. 
Thus  $\mu^*(v)\geq 4+\frac{1}{3}-5\times\frac{1}{2}-1-2\times \frac{1}{3}>0$   
after $v$ sends at most $\frac{1}{2}$ to each incident $3$-face by R4, $1$ to its $2$-neighbour by R1, $\frac{1}{3}$ to each $3$-neighbour by R2.

Let $n_3(v)=3$. Then $v$ is incident to at least two $5^+$-faces containing two $3^-$-neighbours of $v$ since the configuration in Figures  \ref{fig:8-83--38--2} and \ref{fig:8-83-2--38} are reducible.  By R5, each of those faces sends at least $\frac{1}{3}$ to $v$. 
Thus  $\mu^*(v)\geq 4+2\times\frac{1}{3}-5\times\frac{1}{2}-1-3\times \frac{1}{3}>0$  
after $v$ sends at most $\frac{1}{2}$ to each incident $3$-face by R4, $1$ to its $2$-neighbour by R1, $\frac{1}{3}$ to each $3$-neighbour by R2.

Let $n_3(v)=4$. Recall that each $3$-neighbour of $v$ is incident a $4^+$-face containing $v$. Therefore, $v$ is incident to a $(5^+,5^+,8)$-triangle $f$ by Lemma \ref{lem:uv-10}. 
%
%
On the other hand,   $v$ is incident to two $5^+$-faces containing two $3^-$-neighbours of $v$ since the configuration in Figures \ref{fig:8-83--38--2} and \ref{fig:8-83-2--38} are reducible.  By R5, each of those faces sends at least $\frac{1}{3}$ to $v$. Thus $\mu^*(v)\geq 4+2\times \frac{1}{3}-4\times \frac{1}{2}-\frac{1}{3}-1-4\times \frac{1}{3}= 0$  after $v$ sends at most $\frac{1}{2}$ to each incident $3$-face other than $f$ by R4, $\frac{1}{3}$ to  $f$ by R4, $1$ to its $2$-neighbour by R1, $\frac{1}{3}$ to each $3$-neighbour by R2.

$\quad$\textbf{(6.5).} Let $m_3(v)=6$. Note that $vv_1$ is contained in a $3$-face since $G$ has no $4$-fan. Then $v$ is not adjacent to any $3$-vertex by Lemma \ref{lem:8-no-2-on3face-and-3-on3face}.
Hence  $\mu^*(v)\geq 4-6\times\frac{1}{2}-1= 0$  after $v$ sends at most $\frac{1}{2}$ to each incident $3$-face by R4, $1$ to its $2$-neighbour by R1.

\section*{Declarations}

\textbf{Conflict of interest} The authors has no conflicts of interest to declare that are relevant to the content of this article. 

\textbf{Availability of data and material} Data sharing not applicable to this article as no datasets were generated or analysed during the current study.


\end{document}